\documentclass[11pt]{amsart}
\usepackage[applemac]{inputenc}
\usepackage{amsmath}
\usepackage{enumerate}
\usepackage{amssymb}
\usepackage{mathrsfs}
\usepackage{a4wide}
\usepackage{amsthm}
\usepackage{epsfig}
\usepackage{color}
\numberwithin{equation}{section}

\newtheorem{theorem}{Theorem}[section]
\newtheorem{lem}[theorem]{Lemma}
\newtheorem{defi}[theorem]{Definition}
\newtheorem{pro}[theorem]{Proposition}
\newtheorem{coro}[theorem]{Corollary}
\newtheorem{remark}[theorem]{Remark}
\newtheorem{example}[theorem]{Example}

\def\mr{\mathbb{R}}

\def\eps{\varepsilon}

\newcommand{\NN}{\mathbb{N}}

\def\r2{\mathbb{R}^2}

\def\rd{\mathbb{R}^d}
\def\rdrd{\mathbb{R}^d\times\mathbb{R}^d}

\def\div{\mathrm{div}\,}

\newcommand{\xxi}{{\mbox{\boldmath$\xi$}}}

\newcommand{\PP}{\mathscr{P}}

\def\w{\mathcal{W}}

\def\id{\mbox{\boldmath$\mathfrak{i}$}}

\def\hw{\widehat{W}}
\def\tw{\widetilde{W}}

\newcommand{\Cost}{{\mathcal C}}

\def \no#1#2#3 {{\bf #1} (#3), #2.}
\def \eds#1#2#3 {#1, #2, #3.}

\def \nono#1#2#3#4 {{\bf#1} (#2), no.#3, #4.}

\newcommand{\de}{\partial}
\newcommand{\vv}{{\mbox{\boldmath$v$}}}
\newcommand{\eeta}{{\mbox{\boldmath$\eta$}}}
\newcommand{\loc}{{\mathrm{loc}}}
\newcommand{\N}{\mathbb{N}}
\newcommand{\R}{\mathbb{R}}

\title{Gradient flows for non-smooth interaction potentials}

\author{J. A. Carrillo, S. Lisini, E. Mainini}

\address{Jos\'e A. Carrillo, ICREA and Departament de Matem\`a\-ti\-ques,
Universitat Aut\`onoma de Barcelona, E-08193 Bellaterra, Spain. On
leave from: Department of Mathematics, Imperial College London,
London SW7 2AZ, UK.} \email{carrillo@mat.uab.es}

\address{Stefano Lisini,
Dipartimento di Matematica ``F. Casorati'', Universit\`a degli
Studi di Pavia, via Ferrata 1, 27100 Pavia, Italy.}
\email{stefano.lisini@unipv.it}

 \address{ Edoardo Mainini,
 D\'epartement de Math\'ematiques,
UMR 8628 Universit\'e Paris-Sud 11-CNRS,
B\^atiment 425,
Facult\'e des Sciences d'Orsay,
Universit\'e Paris-Sud 11,
F-91405 Orsay Cedex.}
\email{edoardo.mainini@math.u-psud.fr}

\subjclass[2010]{49K20, 35F99}

\begin{document}

\maketitle


\begin{abstract}
We deal with a nonlocal interaction equation describing the
evolution of a particle density under the effect of a general
symmetric pairwise interaction potential, not
necessarily in convolution form.
We describe the case of a convex (or $\lambda$-convex) potential, possibly not smooth at several points,
generalizing the results of \cite{CDFLS}. We also identify the cases in which
the dynamic is still governed by the continuity equation with
well-characterized nonlocal velocity field.
\end{abstract}


\section{Introduction}
Let us consider a distribution of particles, represented by a
Borel probability measure $\mu$ on $\rd$. We introduce the
interaction potential $\mathbf{W}:\rdrd\to\mr$.
The value $\mathbf{W}(x,y)$ describes the interaction of two particles of unit mass at the
positions $x$ and $y$. The
total energy of a distribution $\mu$ under the effect of the
potential is given by the interaction energy functional, defined
by
\begin{equation}\label{W}
    \mathcal{W}(\mu):=\frac12\int_{\rdrd}\mathbf{W}(x,y)\,d(\mu\times\mu)(x,y).
\end{equation}
We assume that $\mathbf{W}$ satisfies the following
assumptions:
\begin{itemize}
\item[\it i)] $\mathbf{W}:\rdrd\to\mr$ is symmetric, i.e.
\begin{equation}\label{sym}
\mathbf{W}(x,y)=\mathbf{W}(y,x) \quad\mbox{for every $x,y\in\rd$};
\end{equation}
\item[\it ii)]  $\mathbf{W}$ is a $\lambda$-convex function for
some $\lambda\in\mr$, i.e.
\begin{equation}\label{convex}
 \mbox{the function}\quad
(x,y)\mapsto\mathbf{W}(x,y)-\frac \lambda 2
(|x|^2+|y|^2)\quad\mbox{is convex};
\end{equation}
\item[\it iii)] $\mathbf{W}$ satisfies the quadratic growth
condition at infinity, i.e. there exists a constant $C>0$ such that
\begin{equation}\label{growth}\mathbf{W}(x,y)\leq C(1+|x|^2+|y|^2).
\end{equation}
\end{itemize}

We are interested in the evolution problem given by the
continuity equation
\begin{equation}\label{main}
    \de_t \mu(t) + \div(\vv(t)\mu(t))=0,\qquad \mu(0)=\mu_0,
\end{equation}
describing the dynamics of the particle density, whose total mass
is conserved, under the mutual attraction-repulsion force given by
\eqref{W}. The velocity vector field $\vv$ enjoys a nonlocal
dependence  on $\mu$. In the basic model represented by a $C^1$
potential $\mathbf{W}$ which depends only on the difference of its
variables, so that we may write $\mathbf{W}(x,y)=W(x-y)$, it is
given by  convolution:
\begin{equation}\label{vf}
\vv(t)=\nabla W\ast\mu(t).
\end{equation}
Under the assumptions \eqref{sym}, \eqref{convex},
\eqref{growth}, in general the function $W$ is not differentiable but only subdifferentiable,
therefore it is reasonable to consider a velocity field of the form
\begin{equation}\label{velocity}
\vv(t)=\eeta(t)\ast\mu(t),
\end{equation}
where $\eeta$ represents a Borel measurable selection in the
subdifferential of $W$, and we will write $\eeta\in\partial W$. In
general, such selection is not independent of $t$. We stress
that $x\mapsto\eeta(x)$ needs to be a pointwise defined object,
since the solutions we consider are probability measures, and
since this model typically presents concentration phenomena when
starting with absolutely continuous initial data.

In this paper, we are going to analyse equations of the form
\eqref{main}-\eqref{velocity} as the gradient flow of the
interaction energy \eqref{W} in the space of Borel probability
measures with finite second moment, endowed with the
metric-differential structure induced by the so-called Wasserstein
distance. This interpretation coming from the optimal transport
theory was introduced in \cite{O} for nonlinear diffusion
equations and generalized for a large class of functionals
including potential, interaction, and internal energy by different
authors \cite{CMV,AGS,CMV2}, see \cite{V,V2} for related
information.

The gradient flow interpretation allows to construct solutions by
means of variational schemes based on the euclidean optimal
transport distance as originally introduced in \cite{JKO} for the
linear Fokker-Planck equation. The convergence of these
variational schemes for general functionals was detailed in
\cite{AGS}. The results in this monograph, which are quickly
summarized in Section \ref{wassersteinsection}, apply to the
interaction equation \eqref{main}-\eqref{vf}, with a $C^1$ smooth
potential verifying the convexity assumption \eqref{convex} and a
growth condition at infinity stricter than \eqref{growth}.

On the other hand, these equations have appeared in the literature
as simple models of inelastic interactions \cite{MY,BCP,BCCP,T} in
which the asymptotic behavior of the equations is given by a total
concentration towards a unique Delta Dirac point measure. The
typical potential in these models was a power law,
$\mathbf{W}(x,y)=|x-y|^\alpha$, $\alpha\geq 0$. Moreover, it was
noticed in \cite{LT} that the convergence towards this unique
steady state was in finite time for certain range of exponents in
the one dimensional case.

Also these equations appear in very simplified swarming or
population dynamics models for collective motion of individuals,
see \cite{MEBS,BL,BCL,KSUB,BCLR} and the references therein. The
interaction potential models the long-range attraction and the
short-range repulsion typical in animal groups. In case the
potential is fully attractive, equation \eqref{main} is usually
referred as the aggregation equation. For the aggregation
equation, finite time blow-up results for weak-$L^p$ solutions,
unique up to the blow-up time, have been obtained in the
literature \cite{BCL,BLR,CR}. In fact, those results conjectured
that solutions tend to concentrate and form Dirac Deltas in finite
time under suitable conditions on the interaction potential. On
the other hand, the confinement of particles is shown to happen
for short-range repulsive long-range attractive potentials under
certain conditions \cite{CDFLS2}. Some singular stationary states
such as uniform densities on spheres have been identified as
stable/unstable for radial perturbations in \cite{BCLR} with sharp
conditions on the potential. Finally, in the one dimensional case,
stationary states formed by finite number of particles and smooth
stationary profiles are found whose stability have been studied in
\cite{FellnerRaoul1,FellnerRaoul2} in a suitable sense.

A global-in-time well-posedness theory of measure weak solutions
have been developed in \cite{CDFLS} for interaction potentials of
the form $\mathbf{W}(x,y)=W(x-y)$ satisfying the assumptions
\eqref{sym},\eqref{convex}, \eqref{growth}, and additionally being
$C^1$-smooth except possibly at the origin. The convexity
condition \eqref{convex} restricts the possible singularities of
the potential at the origin since it implies that $W$ is
Lipschitz, and therefore the possible singularity cannot be worse
than $|x|$ locally at the origin. Nevertheless, for a class of
potentials in which the local behavior at the origin is like
$|x|^\alpha$, $1\leq \alpha < 2$, the solutions converge towards a
Delta Dirac with the full mass at the center of mass of the
solution. The condition for blow-up is more general and related to
the Osgood criterium for uniqueness of ODEs \cite{BCL,CDFLS,BLR}.
Note that the center of mass of the solution is preserved, at
least formally, due to the symmetry assumption \eqref{sym}.

In this work, we push the ideas started in \cite{CDFLS} further in
the direction of giving conditions on the interaction potential to
have a global-in-time well-posedness theory of measure solutions.
The solutions constructed in Section \ref{wassersteinsection} will
be \emph{gradient flow solutions}, as in \cite{AGS}, built via the
variational schemes based on the optimal transport Wasserstein distance. The
crucial point for the analysis in this framework is the
identification of the velocity field in the continuity equation
satisfied by the limiting curve of measures from the approximating
variational scheme. In order to identify it, we need to
characterize the sub-differential of the functional defined in
\eqref{W} with respect to the differential structure induced by
the Wasserstein metric. The Wasserstein sub-differential of the functional $\mathcal{W}$, which is
rigorously introduced in Section \ref{wassersteinsection}, is
defined through variations along transport maps. It turns out that
that the element of minimal norm in this sub-differential, which
will be denoted by $\partial^o\mathcal{W}(\cdot)$, is the element
that governs the dynamics. Actually, it gives the velocity field
via the relation $\vv(t)=-\partial^o\mathcal{W}(\mu(t))$ for a.e. $t \in (0,\infty)$, which
corresponds to the notion of \emph{gradient flow solution}. This
notion will be discussed in Section \ref{wassersteinsection},
where we will give the precise definition and recall from
\cite[Chapter 11]{AGS} the main properties, such as the semigroup
generation.

In Section \ref{characterizationsection}, we give a
characterization of the subdifferential in the general case of the
interaction potential $\mathbf{W}(x,y)$ satisfying only the basic
assumptions \eqref{sym},\eqref{convex}, and \eqref{growth}.
However, the element of minimal norm in the subdifferential is not
fully identified and cannot be universally characterized.
Nevertheless, the global well-posedness of the evolution semigroup
in measures is obtained.

A distinguished role will be played by the case of a kernel
function $\mathbf{W}(x,y)$ which depends only on the difference
$x-y$ of its arguments. Hence we will often consider one of the
following additional assumptions.
\begin{itemize}
    \item[\it iv)] There exists $W:\rd\to\mathbb{R}$ such that
    \begin{equation}\label{difference}\mathbf{W}(x,y)=W(x-y).\end{equation}
\end{itemize}
\begin{itemize}
\item[\it v)] There exists $w:\mathbb{R}\to\mathbb{R}$ such that
\begin{equation}\label{distance}\mathbf{W}(x,y)=W(x-y)=w(|x-y|).\end{equation}
\end{itemize}
The radial hypothesis is frequently made in models, and
corresponds to an interaction between particles which depends only
on their mutual distance vector. In case $\mathbf{W}(x,y)$ is also
radial and convex, we can fully generalize the identification of
the element of minimal norm in the subdifferential of the
interaction energy done in \cite{CDFLS}, regardless of the number
of nondifferentiability points of $W$. We complement our results
with explicit examples showing the sharpness of these
characterizations.


Before to state the results and in order to fix notations we recall the characterization of subdifferential for $\lambda$-convex functions.
Given a $\lambda$-convex function $V:\R^k\to \R$, a vector $\xxi$ belongs to the subdifferential of $V$ at the point $x\in\R^k$ if and only if
\begin{equation}\label{def:subdiff}
    V(z)-V(x) \geq \langle \xxi, z-x \rangle + \frac{\lambda}{2}|z-x|^2 \qquad \forall z\in\R^k,
\end{equation}
and we write $\xxi\in \partial V(x)$.
In this case, for every $x\in\R^k$, we have that
$\partial V(x)$ is a not empty closed convex subset of $\R^k$. We denote by
$\partial^oV(x)$ the unique element of minimal euclidean norm in $\partial V(x)$.

\subsection*{The main results}
Let us give a brief summary of the results contained in this
paper. The main theorem deals with radial-convex potentials and
reads as follows.

\emph{Let $\mathbf{W}$ satisfy the three basic assumptions above:
\eqref{sym}, \eqref{convex}, and \eqref{growth}. If in addition
$\textbf{W}$ satisfies \eqref{difference}, \eqref{distance} and
is convex (that is, $\lambda\geq 0$ in \eqref{convex}), then there exists a
unique \emph{gradient flow solution} to the equation
\begin{equation}\label{strongequation}
\partial_t\mu(t)-\div((\partial^oW\ast\mu(t))\mu(t))=0.
\end{equation}}
This solution is the gradient flow of the energy $\mathcal{W}$, in
the sense that the velocity field in \eqref{strongequation}
satisfies
\[
\partial^oW\ast\mu(t)=\partial^o\mathcal{W}(\mu(t)).
\]

On the other hand, when omitting the radial hypothesis \eqref{distance}, or
when letting the potential be $\lambda$-convex but not convex, we
show that  the evolution of the system under the effect of the
potential, that is the gradient flow of $\mathcal{W}$, is
characterized by \eqref{main}-\eqref{velocity}, where
$\eeta(t)\in$ is a Borel anti-symmetric selection in $\partial W$ for almost every $t$. The corresponding
rigorous statement is found in Section \ref{wassersteinsection}.

About this last result, let us remark that the velocity vector
field is still written in terms of a suitable selection $\eeta$ in
the local subdifferential of $W$, but such selection $\eeta$ is
not in general the minimal one in $\partial W$, and it is not a
priori independent of $t$.
By this characterization we also recover the result of \cite{CDFLS}, where the only non smoothness point for W is the origin: in such case, for any $t$ we are left with  $\eeta(t)(x)=\nabla W(x)$ for $ x\neq 0$ and $\eeta(t)(0)=0$ for $x=0$, by anti-symmetry.
 We stress that, due to the nonlocal structure of the problem, the task of identifying the velocity vector field becomes much more involved when $W$ has several non smoothness points, even if it is $\lambda$-convex. Later in Section
\ref{convolutionsection}, we will analyse some particular
examples, showing that in general it is not possible to write the
velocity field in terms of a single selection in $\partial W$.

Finally, when omitting also the assumption \eqref{difference}, we
break the convolution structure: in this more general case we show
that the velocity is given in terms of an element of
$\partial_1 \mathbf{W}$, or equivalently of $\partial_2\mathbf{W}$
by symmetry, where $\partial_1 \mathbf{W}$ and $\partial_2\mathbf{W}$ denotes the partial subdifferentials of
$\mathbf{W}$ with respect to the first $d$ variables or the last $d$ variables respectively,
\begin{equation*}
\vv(t)(x)=\int_{\rd}\eeta(t)(x,y)\,d \mu(t)(y),
\end{equation*}
where $\eeta(t)\in\partial_1\mathbf{W}$ for almost any $t$. An additional
joint subdifferential condition is also be present in this case,
for the rigorous statement we still address to the next section.

\subsection*{Pointwise particle model and asymptotic behavior}
In the model case of a system of $N$ point particles, discussed in
Section \ref{particlesection}, the dynamics are governed by a
system of ordinary differential equations. In this case
equation \eqref{main}-\eqref{velocity} corresponds to
\begin{equation*}
\frac {dx_i}{dt}=\sum_{j=1}^N m_j\,\eeta(t)(x_j-x_i),\qquad
i=1,\ldots, N,
\end{equation*}
where $x_i(t)$ is the position of the $i$-th particle and $m_i$ is
its mass. It is shown in \cite{CDFLS} that if the attractive
strength of the potential is sufficiently high, all the particles
collapse to the center of mass in finite time. We will remark how
this result is still working under our hypotheses under the same
Osgood criterium as in \cite{BCL,CDFLS} for fully attractive
potentials. For non-convex non-smooth repulsive-attractive
potentials, albeit $\lambda$-convex, the analysis leads to
non-trivial sets of stationary states with singularities that
cannot be treated by the theory in \cite{CDFLS}. Our analysis
shows that a very wide range of asymptotic states is indeed
possible, we give different explicit examples.

\subsection*{Plan of the paper}
In the following Section \ref{wassersteinsection} we introduce the
optimal transport framework and the basic properties of our energy
functional, in particular the subdifferentiability and  the
$\lambda$-convexity along geodesics. We briefly explain what is a
gradient flow in the metric space $\PP_2(\mathbb{R}^d)$ and we
introduce the notion of gradient flow solution.  We present the
general well-posedness result of \cite{AGS} and show how it will
apply to our interaction models, stating our main results. In
Section \ref{characterizationsection}, we make a fine analysis on
the Wasserstein subdifferential of $\mathcal{W}$ and find a first
characterization of its element of minimal norm. In Section
\ref{convolutionsection} we particularize the characterization to
the case of assumption \eqref{difference}, which is the
convolution case. In particular, we have the strongest result in
the case of assumption \eqref{distance}. Section
\ref{particlesection} relates these arguments to the finite time
collapse results in \cite{CDFLS} showing the characterization of
the velocity field for particles. Section \ref{asymptoticsection}
gives examples of non-smooth non-convex repulsive-attractive
potentials, albeit $\lambda$-convex, leading to non-trivial sets
of stationary states. Finally, the Appendix \ref{appendix} is
devoted to recall technical concepts from the differential
calculus in Wasserstein spaces which are needed in Section
\ref{characterizationsection}.


\section{Wasserstein subdifferential and gradient flow of the
interaction energy}\label{wassersteinsection}

\subsection{Optimal transport framework}
We denote by $\PP_2(\rd)$ the space of Borel probability measures
over $\rd$ with finite second moment, i.e., the set of Borel probability
measures $\mu$ such that
$$
\int_{\rd} |x|^2\,d\mu(x)<\infty \, .
$$
The convergence of probability measures is considered in the
narrow sense defined as the weak convergence in the duality
with continuous and bounded functions over $\rd$.
Given $\mu,\nu\in\PP_2(\rd)$ and $\gamma\in \Gamma(\mu,\nu)$, where
\begin{equation*}
\begin{aligned}
\Gamma(\mu,\nu):=\{\gamma\in\PP_2(\rdrd): \gamma(\Omega\times\rd)=\mu(\Omega),\,&\gamma(\rd\times\Omega)=\nu(\Omega),\\
& \text{ for every Borel set }\Omega\subset \rd\},
\end{aligned}
\end{equation*}
the euclidean quadratic transport cost between $\mu$ and $\nu$ with respect to the transport plan $\gamma$
is defined by
\[
\mathcal{C}(\mu,\nu\, ;\gamma)=\left(\int_{\rd\times\rd}|x-y|^2\,d\gamma(x,y)\right)^{1/2}.
\]
The ``Wasserstein distance'' between $\mu$ and $\nu$ is
defined by
\begin{equation}\label{defdw}
d_W(\mu,\nu)=\inf_{\gamma\in\Gamma(\mu,\nu)}\mathcal{C}(\mu,\nu\, ;\gamma).
\end{equation}
It is well known that the $\inf$ in \eqref{defdw} is attained by a
minimizer. The minimizers in \eqref{defdw} are called optimal
plans. We denote by $\Gamma_o(\mu,\nu)\subset\Gamma(\mu,\nu)$ the
set of optimal plans between $\mu$ and $\nu$. It is also well
known that $\mu_n \to \mu$ in $\PP_2(\rd)$ (i.e.
$d_W(\mu_n,\mu)\to 0$) if and only if $\mu_n$ narrowly converges
to $\mu$ and
$\int_{\rd}|x|^2\,d\mu_n(x)\to\int_{\rd}|x|^2\,d\mu(x)$.
The space $\PP_2(\rd)$ endowed with the distance $d_W$ is a complete and separable metric space.
 For all
the details on Wasserstein distance and optimal transportation, we
refer to \cite{AGS,V2}.

We recall the push forward notation for a map
$\mathbf{s}:(\rd)^m\to(\rd)^k$, $m,k\geq 1$, and a measure
$\mu\in\PP((\rd)^m)$: the measure
$\mathbf{s}_\#\mu\in\PP((\rd)^k)$ is defined by
$\mathbf{s}_\#\mu(A)=\mu(\mathbf{s}^{-1}(A))$, where $A$ is a
Borel set. A transport plan $\gamma\in\Gamma(\mu,\nu)$ may be
induced by a map $\mathbf{s}:\mathbb{R}^d\to\mathbb{R}^d$ such that $\mathbf{s}_\#\mu=\nu$.
This means that $\gamma=(\id,\mathbf{s})_\#\mu$, where
$\id:\rd\to\rd$ denotes the identity map over $\rd$ and
$(\id,\mathbf{s}):\rd\to\rdrd$ is the product map. Finally,
$\pi^j$ will stand for the projection map on the $j$-th component
of a product space. Hence, if $\gamma$ is a probability measure
over a product space (for instance a transport plan),
$\pi^j_\#\gamma$ is its $j$-th marginal.

The first properties of the interaction potential functional
$\mathcal{W}$ given by  \eqref{W} are contained in the next Proposition.

\begin{pro}\label{firstproperties}
Under assumptions \eqref{convex} and \eqref{growth}, the
functional $\mathcal{W}$ is lower semicontinuous with respect to
the $d_W$ metric and enjoys the following $\lambda$-convexity
property: for every $\mu,\nu\in\PP_2(\rd)$ and every
$\gamma\in\Gamma(\mu,\nu)$
it holds
\begin{equation}\label{geoconv}
\mathcal{W}(\theta^\gamma(t))\leq(1-t)\mathcal{W}(\mu)+t\mathcal{W}(\nu)-\frac\lambda
2\, t(1-t)\mathcal{C}^2(\mu,\nu;\gamma),
\end{equation}
where $\theta^\gamma$ denotes the interpolating curve
$t\in[0,1]\mapsto\theta^\gamma(t)=((1-t)\pi^1+t\pi^2)_\#\gamma\in\PP_2(\rd)$.
\end{pro}

The lower semicontinuity follows from standard arguments, for the
convexity along interpolating curves we refer to \cite[¤9.3]{AGS}.
In particular, since every constant speed
Wasserstein geodesic is of the form $\theta^\gamma$ where $\gamma$
is an optimal plan, then $\mathcal{W}$ is $\lambda$-convex along
every Wasserstein geodesics. We adapt from \cite{AGS} the
definition of the Wasserstein subdifferential for the
$\lambda$-convex functional $\mathcal W$.

\begin{defi}[\textbf{The Wasserstein subdifferential of $\mathcal{W}$}]\label{subdiffdefinition}
Let $\mu\in\PP_2(\rd)$. We say that the vector field ${\xxi}\in
L^2(\rd,\mu;\rd)$ belongs to $\partial\mathcal{W}(\mu)$, the
Wasserstein  subdifferential of the $\lambda$-convex functional
$\w:\PP_2(\rd)\to (-\infty,+\infty]$ at the point $\mu$, if for
every $\nu\in \PP_2(\rd)$ there exists
$\gamma\in\Gamma_o(\mu,\nu)$ such that
\begin{equation}\label{subdifferentialdefinition}
\w(\nu)-\w(\mu)\geq\int_{\rd}\langle\mathbf{\xxi}(x),
y-x\rangle\,d\gamma(x,y) + \frac\lambda 2\,
\mathcal{C}^2(\mu,\nu\, ;\gamma).
\end{equation}
We say that that $\xxi\in\partial_S\mathcal{W}(\mu)$, the strong
subdifferential of $\mathcal{W}$ at the point $\mu$, if for every
$\nu\in\PP_2(\rd)$ and for every admissible plan
$\gamma\in\Gamma(\mu,\nu)$, \eqref{subdifferentialdefinition}
holds.
\end{defi}

The metric slope of the functional
 $\w$ at the point $\mu$ is defined as follows:
\begin{equation*}
    |\partial\w|(\mu):=\limsup_{\nu\to\mu \mbox{ \rm {\footnotesize in }} \PP_2(\rd)}\frac{\left(\w(\nu)-\w(\mu)\right)^+}{d_W(\nu,\mu)},
\end{equation*}
where $(a)^+$ denotes the positive part of the real number $a$.
Since $\w$ is $\lambda$-convex we know that
\begin{equation}\label{convexslope}
    |\partial\w|(\mu)=\min\left\{\|\xxi\|_{L^2(\rd,\mu;\rd)}:\xxi\in\partial
    \w(\mu)\right\}.
\end{equation}
Moreover, the element realizing the minimal norm in
\eqref{convexslope} is unique and we denote it by
$\partial^o\w(\mu)$ (see \cite[Chapter 10]{AGS}). The element of
minimal norm in the subdifferential plays a {\it crucial role},
since it is known to be the velocity vector field of the evolution
equation associated to the gradient flow of the functional
\eqref{W} under certain conditions as reviewed next.


\subsection{The {gradient flow solution}} As already shown in
\cite{CDFLS}, we are forced to consider a notion of solution which
only assumes that the densities are in fact, measures. Actually,
in case of attractive radial potentials verifying assumptions
\eqref{sym}, \eqref{convex} and \eqref{growth}, i.e.
$\textbf{W}(x,y)=w(|x-y|)$, with $w$ increasing, it was
shown in \cite{BCL,BLR} that weak-$L^p$ solutions blow-up in
finite time. Moreover, these weak-$L^p$ solutions can be uniquely
continued as measure solutions, as proved in \cite{CDFLS}, leading
to a total collapse in a single Delta Dirac at the center of mass
in finite time. Furthermore, particle solutions, i.e, solutions
corresponding to an initial data composed by a finite number of
atoms, remain particle solutions for all times for the evolution
of \eqref{main}. Summarizing we can only expect that a regular
solution enjoys local in time existence.

Our  well-posedness results are based on  the following abstract
theorem for gradient flow solutions. For all the details we refer
to \cite[Theorem 11.2.1]{AGS}, where some more properties of these
solutions are remarked.

Before stating the Theorem, we say that a curve $t\in [0,\infty)\mapsto \mu(t)\in \PP_2(\rd)$ is locally absolutely continuous with locally finite energy, and we denote it by $\mu \in AC^2_{\loc}([0,\infty);\PP_2(\R^d))$, if
the restriction of $\mu$ to the interval $[0,T]$ is absolutely continuous for every $T>0$ and
its metric derivative, which exists for a.e. $t>0$ defined by
\[
|\mu'|(t):=\lim_{s\to t}\frac{d_W(\mu(s),\mu(t))}{|t-s|},
\]
belongs to $L^2(0,T)$ for every $T>0$.

\begin{theorem}\label{gradientflow}
Let $\mathbf{W}$ satisfy the hypotheses \eqref{sym},
\eqref{convex} and \eqref{growth}. For any initial datum
$\mu_0\in\PP_2(\rd)$, there exists a unique curve
$\mu \in AC^2_{\loc}([0,\infty);\PP_2(\R^d))$ satisfying
\begin{align*}
&{\partial_t \mu(t)} + \div (\vv(t) \mu(t)) = 0
\mbox{ in } \mathcal{D}'((0,\infty)\times\R^d),\\
&\vv(t) = -\partial^o\mathcal{W}(\mu(t)),\ \ \ \text{ for a.e. } \,\,t>0,\\
&\|\vv(t)\|_{L^2(\mu(t))}=|\mu'|(t) \ \ \ \text{ for a.e. } \,\,t>0,
\end{align*}
with $\mu(0)=\mu_0$. The energy identity
\begin{equation*}
\int_a^b\int_{\R^d} |\vv(t,x)|^2 \,d\mu(t)(x)\,dt +
\mathcal{W}(\mu(b)) = \mathcal{W}(\mu(a))
\end{equation*}
holds for all $0\leq a\leq b<\infty$. Moreover, the solution is
given by a $\lambda$-contractive semigroup $S(t)$ acting on
$\PP_2(\rd)$, that is $\mu(t)=S(t)\mu_0$ with
\[
d_W(S(t)\mu_0,S(t)\nu_0)\leq e^{-\lambda
t}d_W(\mu_0,\nu_0),\quad\forall\mu_0,\nu_0\in\PP_2(\rd) \,.
\]
\end{theorem}

The unique curve given by Theorem \ref{gradientflow} is called
\emph{gradient flow solution} for equation
\begin{equation}\label{particular}
\partial_t\mu(t)=\div (\partial^o\mathcal{W}(\mu(t))\mu(t)).
\end{equation}
Let us finally remark that weak measure solutions as defined in
\cite{CDFLS} are equivalent to gradient flow solutions as shown
there.

Characterizing the element of minimal norm
$\partial^o\mathcal{W}(\mu)$ is then essential to link the
constructed solutions to the sought equation
\eqref{main}-\eqref{vf} or \eqref{main}-\eqref{velocity}. This
characterization was done in \cite{CDFLS} for potentials
satisfying \eqref{sym}, \eqref{convex}, \eqref{growth} and
\eqref{difference}, being the potential $W$ $C^1$-smooth except
possibly at the origin. Under those assumptions, the authors
identified $\partial^o\mathcal{W}(\mu)$ as $\partial^o W\ast\mu$,
i.e.
$$
\partial^o W\ast\mu (x) = \int_{x\neq y} \nabla W(x-y)\,d\mu(y).
$$

The main results in the present work will show that this
characterization can be generalized to {\it convex} potentials
satisfying assumptions \eqref{sym}, \eqref{growth} and the radial
radial hypothesis \eqref{distance}, regardless of the number of
points of non-differentiability of the potential $W$. In Theorem
\ref{radial}, we show that under those conditions, the formula
$\partial^o\mathcal{W}(\mu)=\partial^o W\ast\mu$ also holds, and
the equation takes the form \eqref{strongequation}, which
generalizes the standard form of the interaction potential
evolution \eqref{main}-\eqref{vf}. In the most general case, i.e.
for general potentials satisfying only \eqref{sym}, \eqref{convex}
and \eqref{growth}, we will obtain a characterization in terms of
generic Borel measurable selections in $\partial \mathbf{W}$.
Therefore, the next two sections are devoted to the study of the
velocity field $\partial^o \mathcal{W}$, in order to apply the
abstract result above to the aggregation equation. Let us state
the results:
\begin{itemize}
    \item[$\bullet$] Let $\mathbf{W}$ satisfy \eqref{sym}, \eqref{convex} and \eqref{growth}. Let $\mu\in\PP_2(\rd)$. There holds
\begin{equation}\label{noconvolutionbis}
\partial^o\mathcal{W(\mu)}=\int_{\rd}\eeta(\cdot,y)\,d\mu(y)
\end{equation}
        for some selection   $\eeta\in\partial_1 \mathbf{W}$ having the form
    \begin{equation}\label{joint}\eeta(x,y)=\frac12 \left(\eeta^1(x,y)+
    \eeta^2(y,x)\right),
\end{equation}
     with the couple $(\eeta^1,\eeta^2)$ belonging
    to the joint subdifferential $\partial \mathbf{W}$. This is
    shown in Theorem \ref{minimalselectionchar}.
    \item[$\bullet$] Under the additional assumption
    \eqref{difference}, let $\mu\in\PP_2(\rd)$.
Then we have
\begin{equation}\label{convolution}
\partial^o\mathcal{W(\mu)}=\eeta\ast\mu
\end{equation}
 for some $\eeta\in\partial W$. This is the
characterization following from Corollary \ref{coro}.
    \item[$\bullet$] Finally, when the further condition
    \eqref{distance} holds, and the potential is convex (not only $\lambda$-convex for a negative $\lambda$) there is
    \begin{equation}\label{convolutionbis}
\partial^o\mathcal{W}(\mu)=\partial^o W \ast\mu
    \end{equation}
    for all $\mu\in\PP_2(\rd)$.
    Here, $\partial^o W$ is the element of minimal norm of the
    subdifferential of $W$.  This is proven in the
    subsequent Theorem \ref{radial}.
\end{itemize}
\begin{remark}\rm
Substituting \eqref{noconvolutionbis}, \eqref{convolution} or
\eqref{convolutionbis} in \eqref{particular} and applying Theorem
\ref{gradientflow}, one obtains the corresponding well-posedness
result. In Section \ref{convolutionsection} we will show that
the selection $\eeta$ appearing in \eqref{convolution} (and thus
also the one in \eqref{noconvolutionbis}) is in general depending
on $\mu$. Therefore, in the case of
\eqref{convolution}, the dynamic will be governed by a velocity
field of the form \eqref{velocity}, where the selection depends in
general on $t$. Similarly for the case of
\eqref{noconvolutionbis}. On the other hand, we stress that a
consequence of the last characterization \eqref{convolutionbis},
when applying Theorem \ref{gradientflow},
is that the selection corresponding to the velocity $\vv(t)$ does not depend on $t$.
\end{remark}
\begin{remark}\rm
The joint subdifferential constraint \eqref{joint} has a natural
interpretation: there is a symmetry in the interaction between
particles (action-reaction law).
\end{remark}



\section{Characterization of the element of minimal norm in the
subdifferential}\label{characterizationsection}

In this section and in the next we analyze  the Wasserstein
subdifferential of $\mathcal{W}$ and we prove the main core
results of this work.

\begin{theorem}\label{inequality}
Consider a Borel measurable selection
$(\eeta^1,\eeta^2)\in\partial \mathbf{W}$,
i.e., $(\eeta^1,\eeta^2):\rdrd\to\rdrd$ is a Borel measurable function such that
$(\eeta^1(x,y),\eeta^2(x,y))\in\partial \mathbf{W} (x,y)$ for every $(x,y)\in\rdrd$.
For any $\mu\in\PP_2(\rd)$, the map
\begin{equation}\label{belongs}
\xxi(x):=\frac12\int_{\rd}(\eeta^1(x,y)+\eeta^2(y,x))\,d\mu(y)
\end{equation}
belongs to $\partial_S\mathcal{W}(\mu)$.
In particular $\partial_S\mathcal{W}(\mu)$ is not empty.
\end{theorem}

\begin{proof}
Since $\mathbf{W}$ is $\lambda$-convex, the subdifferential inequality \eqref{def:subdiff} in this case can be written as follows
\begin{align}\label{eq:30}
\mathbf{W}(y_1,y_2)-&\mathbf{W}(x_1,x_2) \\\nonumber &\geq
\langle(\eeta^1(x_1,x_2),\,\eeta^2(x_1,x_2)),(y_1-x_1,y_2-x_2)\rangle+\frac\lambda
2\, |(y_1-x_1,y_2-x_2)|^2,
\end{align}
for every $(x_1,x_2,y_1,y_2)\in (\rd)^4$. Let
$\mu,\nu\in\PP_2(\rd)$ and $\gamma\in\Gamma(\mu,\nu)$. We show
that inequality \eqref{subdifferentialdefinition} holds.
Considering the measure $\gamma\times\gamma$, we can write
\[
\mathcal{W}(\nu)-\mathcal{W}(\mu)=
\frac12\int_{(\rd)^4}(\mathbf{W}(y_1,y_2)-\mathbf{W}(x_1,x_2))\,d(\gamma\times\gamma)(x_1,y_1,x_2,y_2).
\]
Hence, by \eqref{eq:30},
\begin{align*}
\mathcal{W}(\nu)\!-\!\mathcal{W}(\mu)\geq&\,\frac12
\int_{(\rd)^4}\!\!\left[\langle\eeta^1(x_1,x_2),(y_1-x_1)\rangle
\! + \!
\langle\eeta^2(x_1,x_2),(y_2-x_2)\rangle\right]\!d(\gamma\times\gamma)(x_1,y_1,x_2,y_2)\\
& +\frac\lambda 4
\int_{(\rd)^4}(|y_1-x_1|^2+|y_2-x_2|^2)\,d(\gamma\times\gamma)(x_1,y_1,x_2,y_2).
\end{align*}
The last term is $\frac\lambda 2\,\mathcal{C}^2(\mu,\nu;\gamma)$,
so that a change of variables gives
\begin{align*}
\mathcal{W}(\nu)-\mathcal{W}(\mu)\geq&\,\frac12\int_{\rd}\int_{\rdrd}\langle\eeta^1(x_1,x_2),(y_1-x_1)\rangle
\,d\gamma(x_1,y_1)\,d\mu(x_2)\\\vspace{6pt}
&+\frac12\int_{\rd}\int_{\rdrd}\langle\eeta^2(x_1,x_2),(y_2-x_2)\rangle\,d\gamma(x_2,y_2)\,d\mu(x_1)
+\frac\lambda 2\,\mathcal{C}^2(\mu,\nu;\gamma)\\\vspace{6pt}
=&\,\frac12\int_{\rd}\int_{\rdrd}\langle\eeta^1(x,z),(y-x)\rangle\,d\gamma(x,y)\,d\mu(z)\\\vspace{6pt}
&+\frac12\int_{\rd}\int_{\rdrd}\langle\eeta^2(z,x),(y-x)\rangle\,d\gamma(x,y)\,d\mu(z)
+\frac\lambda 2\,\mathcal{C}^2(\mu,\nu;\gamma)\\
=&\,\int_{\rdrd}\!\left\langle\frac12\int_{\rd}(\eeta^1(x,z)+\eeta^2(z,x))\,d\mu(z),(y-x)\right\rangle
d\gamma(x,y)\!+\!\frac\lambda 2\,\mathcal{C}^2(\mu,\nu;\gamma)
\end{align*}
as desired.
\end{proof}

In the case of a smooth interaction function $\mathbf W$,
there is a complete characterization of the strong subdifferential $\de_S\w(\mu)$
which is single valued.

\begin{pro}[\textbf{The smooth case}]\label{smooth}
Let $\mu\in\PP_2(\rd)$. If $\mathbf{W}\in C^1(\rdrd)$ satisfies
the assumptions \eqref{sym}, \eqref{convex}, and \eqref{growth},
then the strong Wasserstein subdifferential is a singleton and it
is of the form
\begin{equation}\label{smoothgradient}
\begin{aligned}
\partial_S \mathcal{W}(\mu)&=\left\{\int_{\rd}\nabla_1
\mathbf{W}(\cdot,y)\,d\mu(y)\right\}=\left\{\int_{\rd}\nabla_2\mathbf{W}(y,\cdot)\,d\mu(y)\right\}\\
&=\left\{\frac12\int_{\rd}\big(\nabla_1 \mathbf{W}(\cdot,y)+\nabla_2\mathbf{W}(y,\cdot)\big)\,d\mu(y)\right\},
\end{aligned}
\end{equation}
where $\nabla_1$ (resp. $\nabla_2$) are the gradients with respect
to the first-$d$ (second-$d$) variables of $\rdrd$.
\end{pro}
\begin{proof}
Since $\de \mathbf W(x,y)=\{\nabla\mathbf W(x,y)\}$, by Theorem \ref{inequality} and the symmetry
of $\mathbf W$ we have that the right hand sides of \eqref{smoothgradient} are contained
in $\de_S\w(\mu)$.

In order to prove the opposite inclusion, assume that $\xxi\in L^2(\rd,\mu;\rd)$ belongs to $\partial_S\mathcal{W}(\mu)$.
Let $\mathbf{s}\in L^2(\rd,\mu;\rd)$ be an arbitrary vector field,
$\nu=\mathbf{\mathbf{s}}_\#\mu$ and
$\mu_t=(\id+t\mathbf{s})_\#\mu$. Writing
\eqref{subdifferentialdefinition} in correspondence of the plan
\[
\gamma_t=(\id,\id+t\mathbf{s})_\#\mu
\]
between $\mu$ and $\mu_t$, we have
\[
\mathcal{W}(\mu_t)-\mathcal{W}(\mu)\geq\int_{\rdrd}\langle\xxi(x),y-x\rangle\,d\gamma_t(x,y)+\frac\lambda
2\, \Cost^2(\mu,\mu_t;\gamma_t).
\]
Hence, for every $t>0$
\begin{align*}
\frac1{2t}\int_{\rdrd}\left(\mathbf{W}(x+t\mathbf{s}(x),y+t\mathbf{s}(y))-\mathbf{W}(x,y)\right)&
\,d(\mu\times\mu)(x,y) \\ \nonumber & \geq\int_{\rd}\langle
\xxi(x),\mathbf{s}(x)\rangle \,d\mu(x)+\frac{\lambda}
2\,t\|\mathbf{s}\|^2_{L^2(\mu)},
\end{align*}
and, by a direct computation
\begin{align}\label{subdiffmap2}
\frac1{2t}\int_{\rdrd} \big(\mathbf{W}(x+t\mathbf{s}(x),&
\,y+t\mathbf{s}(y))-\mathbf{W}(x,y)\big) d(\mu\times\mu)(x,y)
\\ \nonumber \geq & \,\frac{\lambda}{4t}\int_{\rdrd}
\big(|(x+t\mathbf{s}(x),y+t\mathbf{s}(y))|^2-|(x,y)|^2\big)
d(\mu\times\mu)(x,y)
\\ \nonumber & +\int_{\rd}\langle
\xxi(x),\mathbf{s}(x)\rangle \,d\mu(x)-\lambda \int_{\rd}\langle
x,\mathbf{s}(x)\rangle \,d\mu(x)\,.
\end{align}
Since $\mathbf{W}$ is $\lambda$-convex, the map
\[
t\mapsto
\frac1t\left(\mathbf{W}(x+t\mathbf{s}(x),y+t\mathbf{s}(y))-\mathbf{W}(x,y)\right)-\frac\lambda{2t}\,
\big(|(x+t\mathbf{s}(x),y+t\mathbf{s}(y))|^2-|(x,y)|^2\big)
\]
is nondecreasing in $t$, for $t>0$. Taking advantage of the $C^1$
regularity and the quadratic growth of $\mathbf{W}$, by the
monotone convergence theorem, we can pass to the limit in
\eqref{subdiffmap2} as $t$ goes to $0$, obtaining
\begin{equation}\label{scalardouble}
\frac12\int_{\rdrd}\langle\nabla
\mathbf{W}(x,y),(\mathbf{s}(x),\mathbf{s}(y))\rangle\,d(\mu\times\mu)(x,y)\geq
\int_{\rd}\langle \xxi,\mathbf{s}\rangle\,d\mu.
\end{equation}
Since by the symmetry of $\mathbf{W}$ we have
$\nabla_1\mathbf{W}(x,y)=\nabla_2 \mathbf{W}(y,x)$ for any $x,y\in\rd$,
we can write
\begin{align*}
\frac12\int_{\rdrd}\!\!&\langle\nabla
\mathbf{W}(x,y),(\mathbf{s}(x),\mathbf{s}(y))\rangle\,d(\mu\times\mu)(x,y)\\
=&\frac12\int_{\rdrd}\!\!\!
\langle\nabla_1\mathbf{W}(x,y),\mathbf{s}(x)\rangle\,d(\mu\times\mu)(x,y)+\frac12\int_{\rdrd}\!\!\!
\langle\nabla_1\mathbf{W}(y,x),\mathbf{s}(y)\rangle\,d(\mu\times\mu)(x,y)\\
=&\int_{\rdrd}\!\!\langle\nabla_1
\mathbf{W}(x,y),\mathbf{s}(x)\rangle\,d(\mu\times\mu)(x,y).
\end{align*}
This way, \eqref{scalardouble} becomes
\[
\int_{\rd}\left\langle\int_{\rd}\nabla_1
\mathbf{W}(x,y)\,d\mu(y),\mathbf{s}(x)\right\rangle\,d\mu(x)\geq\int_{\rd}\langle
\xxi,\mathbf{s}\rangle\,d\mu,
\]
that is
\[
\int_{\rd}\left\langle\int_{\rd}\nabla_1
\mathbf{W}(x,y)\,d\mu(y)-\xxi(x),\mathbf{s}(x)\right\rangle\,d\mu(x)\geq 0.
\]
Since $\mathbf{s}\in L^2(\rd,\mu;\rd)$ is arbitrary we conclude
that $\xxi(x)=\int_{\rd}\nabla_1\mathbf{W}(x,y)\,d\mu(y)$ as
elements of $L^2(\rd,\mu;\rd)$.
\end{proof}

When we drop the assumption of $\mathbf{W}\in C^1(\rd\times\rd)$
the subdifferential of $\w$ is in general multivalued. In the next
Theorem, we characterize the element of minimal norm in the
sub\-di\-fferential of $\w$ at the point $\mu$, which is of the
form \eqref{belongs}.

\begin{remark}\rm
The property stating that every element of the subdifferential of
$\w$ at the point $\mu$ is of the form \eqref{belongs} for a
suitable Borel selection of the subdifferential of $\mathbf{W}$
could be in general very difficult and we do not know if it is
true.
\end{remark}

\begin{theorem}\label{minimalselectionchar}
Let $\mu\in\PP_2(\rd)$ and $\xxi=\partial^o\w(\mu)$. Then there exists a measurable
selection $(\eeta^1,\,\eeta^2)\in\partial \mathbf{W}$ such that
\begin{equation*}
\xxi(x)=\frac12\int_{\rd}(\eeta^1(x,y)+\eeta^2(y,x))\,d\mu(y).
\end{equation*}
\end{theorem}

The proof of Theorem \ref{minimalselectionchar} needs several
preliminary results. We will make use of a sequence of regularized
functionals $\mathcal{W}_n$. We recall that the  {\it
Moreau-Yosida} approximation of the function $\mathbf{W}$ is
defined as
\begin{equation}\label{moreau}
\mathbf{W}_n(x,y):=\inf_{(v,w)\in\rdrd} \Big\{\mathbf{W}(v,w)+\frac
n2\,\left|(x-v,y-w)\right|^2\Big\}.
\end{equation}
We have $\mathbf{W}_n(x,y)\leq\mathbf{W}(x,y)$ for every $(x,y)\in\rd\times\rd$,
$\mathbf{W}_n\in C^{1,1}(\rd\times\rd)$ and the sequence $\{\mathbf{W}_n\}_{n\in\N}$
converges pointwise and monotonically
to $\mathbf{W}$ as $n\to\infty$.

If the $\lambda$-convexity property of $\mathbf{W}$ is satisfied,
there exists a constant $K>0$ such that
\begin{equation}\label{lambdaconvexity}
\mathbf{W}(x,y)\geq -K(1+|x|^2+|y|^2),
\end{equation}
and in this case we can show the corresponding bound for
$\mathbf{W}_n$, uniformly in $n$.

\begin{pro}\label{negativequadratic} If $\mathbf{W}$ enjoys the $\lambda$-convexity
assumption \eqref{convex}, there exist $\bar{n}\in\mathbb{N}$ and
a constant $\bar{K}>0$ such that for all $n>\bar{n}$
\[
\mathbf{W}_n(x,y)\geq -\bar{K}(1+|x|^2+|y|^2)\, .
\]
\end{pro}
\begin{proof}
 Indeed, by \eqref{moreau} and
the  estimate \eqref{lambdaconvexity} there holds
\begin{equation}\label{moreauestimate}
\mathbf{W}_n(x,y)\geq\inf_{(v,w)\in\rdrd}
\Big\{-K(1+|v|^2+|w|^2)+\frac n2\,\left|(x-v,y-w)\right|^2\Big\}.
\end{equation}
Here we compute the minimum. The gradient of the right hand side
is \[(-2Kv-n(x-v),-2Kw-n(y-w)),\] and it vanishes when
$(v,w)=\frac{n}{n-2K}\,(x,y)$. Substituting in
\eqref{moreauestimate} we get
\[
\mathbf{W}_n(x,y)\ge
-K\left(1+\frac{n^2}{(n-2K)^2}\,(|x|^2+|y|^2)+\frac{4nK^2}{2(n-2K)^2}\,(|x|^2+|y|^2)\right).
\]
If we chose for instance $\bar{K}=2K$, it is then clear that there
exists a large enough $\bar{n}$ (depending only on $K$) such that
the desired inequality holds for any $n>\bar{n}$.
\end{proof}

We define the approximating interaction functionals
\begin{equation}\label{approxfunc}
\mathcal{W}_n(\mu):=\frac12\int_{\rdrd}\mathbf{W}_n(x,y)\,d(\mu\times\mu)(x,y).
\end{equation}
\begin{remark}[\textbf{Semicontinuity properties of $\mathcal{W}$}]\label{lsc}\rm
Since $\mathbf{W}$ might enjoy a negative quadratic behavior at
infinity, it is not true that $\mathbf{W}$ is lower semicontinuous
also with respect to the narrow convergence. By the way, it is
shown in \cite[¤2]{CDFLS} that one can choose $\tau_0$ small
enough (depending only on $\mathbf{W}$) such that for any
$\tau<\tau_0$ and for any $\mu\in\PP_2(\rd)$, the functional
\begin{equation}\label{moreauwasserstein}
\nu\mapsto \mathcal{W}(\nu)+\frac1{2\tau}\,d_W^2(\nu,\mu),
\end{equation}
is lower semicontinuous with respect to the narrow convegence.
Moreover, for $\tau<\tau_0$, minimizers do exist for
\eqref{moreauwasserstein}. The arguments of \cite{CDFLS} are given
for a function $\mathbf{W}$ such that $\mathbf{W}(x,x)=0$ for any
$x$ and assumption {\it iv)} holds. They can be adapted in a
straightforward way if these hypotheses are omitted. Moreover, in
the case of the approximating functionals $\mathcal{W}_n$ defined
in \eqref{approxfunc}, we can choose $\tau_0$  independently of
$n$. Indeed, since Proposition \ref{negativequadratic} gives the
bound $\mathbf{W}_n(x,y)\ge -\bar{K}(1+|x|^2+|y|^2)$, for some
$\bar{K}>0$, it is enough to choose $\tau_0$ small enough such
that
\[
\nu\mapsto\int_{\rdrd}-\bar{K}(1+|x|^2+|y|^2)\,d\nu\times\nu(x,y)+\frac1{2\tau}\,d_W^2(\mu,\nu)
\]
is narrowly lower semicontinuous.
\end{remark}

We prove the following more general lower semicontinuity property.

\begin{pro}\label{gammaconvergenceproposition} Let $\tau_0$ be
as in {\rm Remark \ref{lsc}}. Let $\tau<\tau_0$ and let
$\nu\in\PP_2(\rd)$. For any $\mu\in\PP_2(\rd)$ and for any
sequence $\{\mu_n\}_{n\in\N}$ such that $\mu_n$ narrowly converges
to $\mu$ and $\sup_n\int_{\rd}|x|^2\,d\mu_n(x)<+\infty$, there
holds
\[
\mathcal{W}(\mu)+\frac1{2\tau}\,d_W^2(\mu,\nu)\leq\liminf_{n}\left(\mathcal{W}_n(\mu_n)+\frac1{2\tau}\,d_W^2(\mu_n,\nu)\right).
\]
Moreover, for $\mu\in\PP_2(\rd)$ there holds
$\mathcal{W}_n(\mu)\to\mathcal{W}(\mu)$.
\end{pro}
\begin{proof}
Since $\mathbf{W}_n\ge\mathbf{W}_k$ if $n\ge k$, by Remark
\ref{lsc} we have
\[\begin{aligned}
\liminf_{n\to\infty}\left(\mathcal{W}_n(\mu_n)+\frac1{2\tau}\,d_W^2(\mu_n,\nu)\right)&\ge
\liminf_{n\to\infty}\left(\mathcal{W}_k(\mu_n)+\frac1{2\tau}\,d_W^2(\mu_n,\nu)\right)\\&\ge
\mathcal{W}_k(\mu)+\frac1{2\tau}\,d_W^2(\mu,\nu)
\end{aligned}\]
for any fixed $k\in\mathbb{N}$. Now we shall pass to the limit as
$k\to\infty$. Notice that $\mathbf{W}_k \nearrow \mathbf{W}$
pointwise and monotonically, and thus by the monotone convergence
theorem, $\mathcal{W}_k(\mu)$ converges to $\mathcal{W}(\mu)$.
Both statements are proven.
\end{proof}

We recall a suitable notion of convergence of a sequence of vector
fields $\xxi_n\in L^2(\rd,\mu_n;\rd)$.
\begin{defi}\label{strongconvergencedefinition}
Let $\mu_n$ narrowly convergent to $\mu$ and let $\xxi_n\in
L^2(\rd,\mu_n;\rd)$.
We say that $\xxi_n$ weakly converges to
$\xxi\in L^2(\rd,\mu;\rd)$ if
\begin{equation}\label{wc}
\int_{\rd}\langle \xxi_n,\zeta\rangle\,d\mu_n\to\int_{\rd}\langle
\xxi,\zeta\rangle\,d\mu, \qquad \forall \zeta\in
C^\infty_c(\rd;\rd).
\end{equation}
We say that the convergence is strong if \eqref{wc} holds and
\[
\int_{\rd}|\xxi_n|^2\,d\mu_n\to\int_{\rd}|\xxi|^2\,d\mu.
\]
\end{defi}
\begin{remark}\label{metrizable} \rm
Consider the set $\mathcal{M}_d$ of vector measures of the form
$\xxi\mu$, with $\mu\in\PP_2(\rd)$ and $\xxi\in L^2(\rd,\mu;\rd)$.
Reasoning as in \cite[¤5.1]{AGS}, we know that the
weak convergence \eqref{wc} is metrizable on every subset
$\mathcal{A}$ of $\mathcal{M}_d$ such that
\[
\sup_{\mathcal{A}}\int_{\rd}|\xxi|\,d\mu<+\infty.
\]
Moreover, by \cite[Theorem 5.4.4]{AGS}, if
$\mathcal{A}\subset\mathcal{M}_d$ is such that
\[
\sup_{\mathcal{A}}\int_{\rd}|\xxi|^2\,d\mu<+\infty,
\]
then $\mathcal{A}$ is also compact with respect to the weak
convergence \eqref{wc}.
\end{remark}

We also need to define the barycentric projection.

\begin{defi}[\textbf{Disintegration and barycenter}]\label{barycenterdefinition}
Given $\beta\in\Gamma(\mu,\nu)$, we denote by $\beta_x$
the Borel family of measures over $\PP(\rd)$
such that $\beta=\int_{\rd}\beta_x\,d\mu(x)$,
which disintegrates $\beta$ with respect to $\mu$.
The notation above means that the
integral of a Borel function $\varphi:\rd\times\rd\to\R$ such that $\varphi \in L^1(\beta)$, can be sliced as
\[
\int_{\rdrd}\varphi(x,y)\,d\beta(x,y)=\int_{\rd}\int_{\rd}\varphi(x,y)\,d\beta_x(y)\,d\mu(x).
\]
The barycentric projection of $\beta\in\Gamma(\mu,\nu)$ is defined by
\begin{equation*}
    \bar\beta(x):=\int_{\rd} y\,d\beta_x(y).
\end{equation*}
\end{defi}
For more detail about disintegration see \cite[Theorem 2.28]{AFP}.

We can prove the following simple

\begin{pro}\label{planconvergence}
Let $\{\mu_n\}\subset\PP_2(\rd)$, $\{\nu_n\}\subset\PP_2(\rd)$
be sequences with uniformly bounded second moments and narrowly convergent to $\mu$
and $\nu$ respectively.
For every choice $\gamma_n\in\Gamma_o(\mu_n,\nu_n)$,
we have that the sequence $\{\gamma_n\}$ is tight and every limit point
with respect to the narrow convergence in $\PP(\rdrd)$
belongs to $\Gamma_o(\mu,\nu)$.
Moreover, if $\gamma$ is a limit point
and $\gamma_{n_k}$ is a subsequence narrowly convergent to $\gamma$,
then
\begin{equation*}
\bar{\gamma}_{n_k}\to\bar{\gamma}\quad\mbox{weakly in the sense of
{\rm Definition \ref{strongconvergencedefinition}} as }k\to +\infty.
\end{equation*}
\end{pro}
\begin{proof}
The tightness and optimality are contained in \cite[Proposition
7.1.3]{AGS}.  Let $(\gamma_n)_x$ be the disintegration of
$\gamma_n$ with respect to $\mu_n$ and $\gamma_x$ be the
disintegration of $\gamma$ with respect to $\mu$.
Let $\zeta\in C^\infty_c(\rd;\rd)$ and
$f:\rdrd\to\R$ the function defined by $f(x,y)=\langle \zeta(x),y \rangle$. Since $f$ is continuous and
satisfies $|f(x,y)|\leq C|y|$ for every $(x,y)\in\rdrd$ and $\sup_n \int_{\rdrd}(|x|^2 + |y|^2)\, d\gamma_n(x,y)<+\infty$, by
\cite[Lemma 5.1.7]{AGS} we have that $\int_{\rdrd}f(x,y)\, d\gamma_n(x,y) \to \int_{\rdrd}f(x,y)\, d\gamma(x,y)$ as $n\to +\infty$.
Using this property and the definition of barycenter, we have
\[\begin{aligned}
\int_{\rd}\langle \zeta ,\bar{\gamma}_n\rangle \,d\mu_n&=\int_{\rd}\langle \zeta(x) , \int_{\rd}y\,d(\gamma_n)_x(y) \rangle \,d\mu_n(x)\\
&= \int_{\rdrd}\langle \zeta(x),y\rangle\,d\gamma_n(x,y)\to
\int_{\rdrd}\langle \zeta(x),y\rangle \,d\gamma(x,y)\\
&=\int_{\rd}\langle \zeta(x),\int_{\rd}y\rangle \,d\gamma_x(y)\,d\mu(x)
=\int_{\rd}\langle \zeta,\bar{\gamma}\rangle \,d\mu.
\end{aligned}
\]
\end{proof}

We recall a definition from \cite[Chapter 10]{AGS}.

\begin{defi}[\textbf{Rescaled plan}]\label{rescaled} Let
$\tau<\tau_0$ (as in {\rm Remark \ref{lsc}}). Given
$\mu\in\PP_2(\rd)$, let
\[
\mu_\tau=\mathrm{argmin}
\left\{\mathcal{W}(\nu)+\frac1{2\tau}d_W(\nu,\mu):\nu\in\PP_2(\rd)\right\}.
\]
Given $\hat{\gamma}_\tau\in\Gamma_o(\mu_\tau,\mu)$, we define the
rescaled plan as
\[
\gamma_\tau:=\left(\pi^1,\frac{\pi^2-\pi^1}{\tau}\right)_\#\hat{\gamma}_\tau.
\]
\end{defi}

Next we introduce an abstract result about approximation of the
minimal selection in the subdifferential of $\mathcal{W}$. The
argument is indeed a direct consequence of the analysis in
\cite[¤10.3]{AGS}, but requires the concept of plan
subdifferential. Since this is a technical definition, we prefer
to postpone a discussion at the end of the paper. Therefore, the
proof of the following proposition is given in the appendix.

\begin{pro}\label{convergetominimalselection}
Let $\mu$, $\mu_\tau$ and $\gamma_\tau$ be as in {\rm Definition
\ref{rescaled}}. Then $\mu_\tau\to\mu$ in $\PP_2(\rd)$ as $\tau\to
0$. Moreover, denoting by $\bar\gamma_\tau$ the barycenter of
$\gamma_\tau$, we have that $\bar\gamma_\tau\in\partial_S
\mathcal{W}(\mu_\tau)$ and
\[
\bar{\gamma}_\tau\to\partial^o\mathcal{W}(\mu)\quad\mbox{strongly
in the sense of {\rm Definition \ref{strongconvergencedefinition}}
as }\tau\to 0.
\]
\end{pro}

Making use of Proposition \ref{convergetominimalselection} we can
prove the following

\begin{lem}\label{prop:gamma conv} Let $\mu\in\PP_2(\rd)$. There exists a  sequence
$\{\mu_n\}_{n\in\NN}\subset\PP_2(\rd)$, with
$\int_{\rd}|x|^2\,d\mu_n<+\infty$, such that $\mu_n$ narrowly
converges to $\mu$ and
\[
\int_{\rd}\nabla_1
\mathbf{W}_n(\cdot,y)\,d\mu_n(y)\to\partial^o\mathcal{W}(\mu)\quad\mbox{weakly
in the sense of {\rm Definition
\ref{strongconvergencedefinition}}.}
\]
\end{lem}

\begin{proof}
Let $\tau_0$ be as in Remark \ref{lsc},  and consider a measure
$\mu\in\PP_2(\rd)$. Define, for $\tau\leq\tilde{\tau}<\tau_0$,
\[
\mu^h_\tau=\mathrm{argmin}
\left\{\mathcal{W}_h(\nu)+\frac1{2\tau}\,d_W^2(\nu,\mu):\nu\in\PP_2(\rd)\right\}.
\]
By the uniform estimates of \cite[Lemma 2.2.1]{AGS},
\begin{equation}\label{unif}
\sup_{h\in\mathbb{N},\,\tau\leq\tilde{\tau}}\int_{\rd}|x|^2\,d\mu_\tau^h<+\infty.
\end{equation}
This shows that the sequence $\{\mu_\tau^h\}_{h\in\NN}$ is tight
and has bounded second moments. Let $\mu_\tau$ be a narrow limit
point. Proposition \ref{gammaconvergenceproposition} yields, for
any $\nu\in\PP_2(\rd)$,
\[
\begin{aligned}
\mathcal{W}(\mu_\tau)+\frac1{2\tau}\,d_W^2(\mu_\tau,\mu)
&\leq
\liminf_{h\to\infty}\left(\mathcal{W}_h(\mu_\tau^h)+\frac1{2\tau}\,d_W^2(\mu_\tau^h,\mu)\right)\\
&\leq
\liminf_{h\to\infty}\left(\mathcal{W}_h(\nu)+\frac1{2\tau}\,d_W^2(\nu,\mu)\right)\\
&=\mathcal{W}(\nu)+\frac1{2\tau}\,d_W^2(\nu,\mu).
\end{aligned}
\]
This shows that $\mu_\tau$ is indeed a minimizer for
\eqref{moreauwasserstein}. Let
$\hat{\gamma}_\tau^h\in\Gamma_o(\mu_\tau^h,\mu)$ and let
$\gamma_\tau^h$ be the corresponding rescaled plans (see
Definition \ref{rescaled}). If
$\{h(m)\}_{m\in\NN}\subset\mathbb{N}$ is a sequence such that
$\mu_\tau^{h(m)}$ narrowly converges to $\mu_\tau$, since
$\{\mu_\tau^{h(m)}\}_{m\in\NN}$ has uniformly bounded second
moments due to \eqref{unif}, by Proposition \ref{planconvergence}
we have (possibly on a further subsequence that we do not relabel)
the narrow $\PP(\rdrd)$ limits
\[
\hat{\gamma}_{\tau}^{h(m)}\rightharpoonup
\hat{\gamma}_{\tau},\qquad
{\gamma}_{\tau}^{h(m)}\rightharpoonup{\gamma}_{\tau}
\]
and also $\bar{\gamma}_\tau^{h(m)}\to\bar{\gamma}_\tau$ weakly in
the sense of Definition \ref{strongconvergencedefinition}. Let
$\mathbf{d}$ be a distance which metrizes this weak convergence.
Indeed, it is metrizable thanks to Remark \ref{metrizable}
since
\begin{align*}
\int_{\rd} |\bar{\gamma}_\tau^{h(m)}(x)|^2 \,d\mu_\tau^{h(m)}(x) &
\,\leq \int_{\rd} \int_{\rd} |y|^2 \,d(\gamma_\tau^{h(m)})_x (y)
\,d\mu_\tau^{h(m)}(x) = \int_{\rd\times\rd} |y|^2
\,d\gamma_\tau^{h(m)} (x,y)
\\
&\, = \int_{\rd} \frac{|y-x|^2}{\tau^2} d\hat\gamma_t^{h(m)}(y) =
\frac1{\tau^2} d_W^2(\mu_\tau^{h(m)},\mu)
\end{align*}
which is uniformly bounded in $m$ for any fixed $0<\tau\leq
\tilde\tau$ due to \eqref{unif}.

Then, if $\tau(n)$ is a vanishing sequence, we can extract a
further subsequence $\{h(n)\}_{n\in\NN}$ from $\{h(m)\}_{m\in\NN}$
such that,
 \[
\mathbf{d}(\bar{\gamma}_{\tau(n)}^{h(n)},\bar{\gamma}_{\tau(n)})
<\frac1n.
 \]
Then we have
\[\begin{aligned}
\mathbf{d}(\bar{\gamma}_{\tau(n)}^{h(n)},\partial^o\mathcal{W}(\mu))&\leq
\mathbf{d}(\bar{\gamma}_{\tau(n)}^{h(n)},\bar{\gamma}_{\tau(n)})+
\mathbf{d}(\bar{\gamma}_{\tau(n)},\partial^o\mathcal{W}(\mu))\\&\leq
\frac1n+
\mathbf{d}(\bar{\gamma}_{\tau(n)},\partial^o\mathcal{W}(\mu))
\end{aligned}
\]
Invoking Proposition \ref{convergetominimalselection}, we know
 that $\bar{\gamma}_\tau$ converges to
 $\partial^o\mathcal{W}(\mu)$ weakly in the sense of Definition
 \ref{strongconvergencedefinition} as $\tau\to 0$.
  Hence, passing to the limit as $n\to\infty$, we see that
$\bar{\gamma}_n:=\bar{\gamma}_{\tau(n)}^{h(n)}$ weakly converge to
$\partial^o\mathcal{W}(\mu)$ in the sense of Definition
\ref{strongconvergencedefinition}.

Finally, by Proposition \ref{convergetominimalselection} we know
that, for any $n$,
$\bar{\gamma}_n\in\partial_S\mathcal{W}_n(\mu_n)$. Since
$\mathbf{W}_n$ is $C^{1,\,1}$, by the characterization of strong
subdifferential of Proposition \ref{smooth} we have
\[
\bar{\gamma}_n(x)=\int_{\rd}\nabla_1\mathbf{W}_n(x,y)\,d\mu(y)
\]
and the proof is concluded.
\end{proof}

\begin{proof}[\textbf{{Proof of {\rm \textbf{Theorem {\ref{minimalselectionchar}}}}}}]
Let $\mu\in\PP_2(\rd)$, $\xxi=\partial^o\w(\mu)$ and $\mu_n$ the sequence given by Lemma \ref{prop:gamma conv}.
By Proposition \ref{smooth}, the only
element of $\partial_S \mathcal{W}_n(\mu_n)$ is given by
\[
\xxi_n(x):=\int_{\rd} \nabla_1 \mathbf{W}_n(x,y)\,d\mu_n(y).
\]
Let us consider the maps $(\id,\nabla
\mathbf{W}_n):\rd\times\rd\longrightarrow (\rd)^4$ given by
$$
(\id,\nabla \mathbf{W}_n)(x,y)=(x,y,\nabla \mathbf{W}_n(x,y)).
$$
Introducing the measures
$$
\nu_n:=(\id,\nabla \mathbf{W}_n)_\#(\mu_n\times\mu_n),
$$
by Lemma \ref{prop:gamma conv} we have, for any $\zeta\in
C^\infty_0(\rd)$
\begin{equation}\label{weakconvergence}\begin{aligned}
\int_{(\rd)^4} \langle v,\zeta(x)\rangle \,d\nu_n(x,y,v,w)
&=\int_{\rdrd}\langle\nabla_1
\mathbf{W}_n(x,y),\zeta(x)\rangle\,d(\mu_n\times\mu_n)(x,y)\\
&=\int_{\rd}\langle \xxi_n,\zeta\rangle\,d\mu_n
\longrightarrow \int_{\rd}
\langle \xxi,\zeta\rangle\,d\mu.
\end{aligned}
\end{equation}
The sequence $\nu_n$ is tight. Indeed, from the quadratic growth
of $\mathbf{W}$ and $\mathbf{W}_n$ at infinity and the boundedness
of $\mu_n$ in $\PP_2(\rd)$ we obtain the uniform estimate
\[
\sup_{n\in\N}\int_{\rd} |x|^2\,d(\nabla_1
{\mathbf{W}_n}_\#\mu_n)(x)=\sup_n\int_{\rd}|\nabla_1
\mathbf{W}_n|^2\,d\mu_n(x)<+\infty,
\]
which implies the tightness of the marginals of $\nu_n$. Then  we
can extract a subsequence (that we do not relabel) narrowly
converging to some $\nu\in\PP((\rd)^4)$. Moreover, for $\zeta\in
C^\infty_0(\rd)$, due to the linear growth of the integrand we
have
\[
\lim_{n\to\infty}\int_{(\rd)^4} \langle v,\zeta(x)\rangle
\,d\nu_n(x,y,v,w)=\int_{(\rd)^4} \langle v,\zeta(x)\rangle
\,d\nu(x,y,v,w).
\]
The narrow convergence of measures implies that $\mathrm{supp}(\nu)$ is
contained in the Kuratowski minimum limit of the supports of
$\nu_n$ (see for instance \cite[Proposition 5.1.8]{AGS}), i.e. for every $(x,y,\eeta)\in \mathrm{supp}(\nu)$ there exists a sequence
$(x_n,y_n,\eeta_n)\in \mathrm{supp}(\nu_n)$ such that $(x_n,y_n,\eeta_n)$ converges to $(x,y,\eeta)$.
Since, by definition of $\nu_n$, $\mathrm{supp}(\nu_n)\subset \mathrm{graph}(\partial
\mathbf{W}_n)$, then $\mathrm{supp}(\nu)\subset\mathrm{graph}(\partial\mathbf{W})$.
Indeed $\eeta_n\in\partial \mathbf{W}_n(x_n,y_n)$ and passing to the limit in the subdifferential inequality we obtain that
$\eeta\in\partial \mathbf{W}(x,y)$.

Disintegrating $\nu$ with respect to $(x,y)$, we obtain the
measurable family of measures $(x,y)\mapsto \nu_{x,\,y}$ such that
\[
\nu=\int_{\rdrd}\nu_{x,\,y}\,d(\mu\times\mu)(x,y).
\]
It follows that
$\mathrm{supp}(\nu_{x,\,y})\subset\partial\mathbf{W}(x,y)$. In the
limit
\[
\begin{aligned}
\lim_{n\to\infty}\int_{\rd} \langle
\xxi_n(x),\zeta(x)\rangle&\,d\mu_n(x)=\int_{(\rd)^4}
\langle v,\zeta(x)\rangle\,d\nu(x,y,v,w)\\
&=\frac12\int_{(\rd)^4} \langle
v,\zeta(x)\rangle\,d\nu(x,y,v,w)+\frac12\int_{(\rd)^4}
\langle w,\zeta(y)\rangle\,d\nu(x,y,v,w)\\
&=\frac12
\int_{\rd}\left\langle\int_{\rd}\int_{\rdrd}v\,d\nu_{x,\,y}(v,w)\,d\mu(y),\zeta(x)\right\rangle\,d\mu(x)\\
&\quad+\frac12\int_{\rd}\left\langle\int_{\rd}\int_{\rdrd}w\,d\nu_{x,\,y}(v,w)\,d\mu(x),\zeta(y)\right\rangle\,d\mu(y)\\
&=\int_{\rd}\left\langle\int_{\rd}\int_{\rdrd}\,
\frac12\,\left(v\,d\nu_{x,\,y}+\,w\,d\nu_{y,\,x}\right)(v,w)\,d\mu(y),\zeta(x)\right\rangle\,d\mu(x).
\end{aligned}
\]
Here we made use of the symmetry of $\mathbf{W}$, which entails
\[
\int_{(\rd)^4}\phi(v)\varphi(x)\,d\nu(x,y,v,w)=\int_{(\rd)^4}\phi(w)\varphi(y)\,d\nu(x,y,v,w)
\]
for any functions $\phi,\varphi$, by using this latter equality for $\nu_n$ (since
$\mathbf{W}_n$ is symmetric) and passing it to the limit. Defining
\[
\eeta^1(x,y):=\int_{\rdrd}v\,d\nu_{x,\,y}(v,w),\quad\mbox{and}\quad
\eeta^2(x,y):=\int_{\rdrd}w\,d\nu_{x,\,y}(v,w),
\]
we obtain
\[
\lim_{n\to\infty}\int_{\rd} \langle
\xxi_n(x),\zeta(x)\rangle\,d\mu_n(x)=\int_{\rd}\left\langle\int_{\rd}\frac12\,\left(\eeta^1(x,y)+\eeta^2(y,x)
\right)\,d\mu(y),\zeta(x)\right\rangle\,d\mu(x).
\]
By \eqref{weakconvergence}, we get
\[
\xxi(x)=\int_{\rd}\frac12\,\left(\eeta^1(x,y)+\eeta^2(y,x)
\right)\,d\mu(y)
\]
On the other hand, we proved that
$\mbox{supp}(\nu_{x,\,y})\subset\partial\mathbf{W}(x,y)$. Since
$\eeta^1(x,y),\eeta^2(x,y)$ are barycenters and
$\partial\mathbf{W}(x,y)$ is convex, we have that
$\left(\eeta^1(x,y),\eeta^2(x,y)\right)\in\partial\mathbf{W}(x,y)$.
\end{proof}


\section{The convolution case}\label{convolutionsection}

\subsection{The case of $\mathbf{W}$ depending on the difference}

In the  case of  assumption \eqref{difference} we can
particularize the results above.
\begin{lem}\label{lem:sd}
Let \eqref{difference} hold. $(\eeta^1,\eeta^2)\in\de\mathbf{W}$
if and only if there exists $\eeta\in\de W$,  such that
$(\eeta^1,\eeta^2)=(\eeta,-\eeta)$.
\end{lem}
\begin{proof}
Assume that $W$ is convex, the general case follows considering $x\mapsto W(x)-\frac{\lambda}{2}|x|^2$.\\
If $\eeta\in\partial W$  we have for every $(\tilde x,\tilde
y)\in\rdrd$
\[
\begin{aligned}
\mathbf{W}(\tilde{x},\tilde{y})-\mathbf{W}(x,y)=W(\tilde{x}-\tilde{y})-W(x-y)
&\geq \left\langle\eeta(x-y),\tilde{x}- \tilde{y}-(x-y)\right\rangle\\
&=\left\langle\left(\eeta(x-y),-\eeta(x-y)\right),(\tilde{x},\tilde{y})-(x,y)\right\rangle,
\end{aligned}
\]
which means that $(\eeta(x-y),-\eeta(x-y))\in\de\mathbf{W}(x,y)$
by making the abuse of notation $\eeta(x,y)\equiv\eeta(x-y)$.

On the other hand, if $(\eeta^1,\,\eeta^2)\in\partial \mathbf{W}$,
then for every $(\tilde x, \tilde y)\in\rdrd$ we have
\begin{align}
W(\tilde{x}-\tilde{y})-W(x-y)=\mathbf{W}(\tilde{x},\tilde{y})-\mathbf{W}(x,y)&\geq
\left\langle\left(\eeta^1(x,y),\eeta^2(x,y)\right),\left((\tilde{x},\tilde{y})-(x,y)\right)\right\rangle\nonumber\\
&=\langle\eeta^1(x,y),(\tilde{x}-x)\rangle+\langle\eeta^2(x,y),(\tilde{y}-y)\rangle.
\label{sd}
\end{align}
Assuming in particular that $x-y=\tilde x -\tilde y$ the inequality above
reduces to
\[
0\geq \langle\eeta^1(x,y)+\eeta^2(x,y),(\tilde{y}-y)\rangle,
\]
and the arbitrariness of  $\tilde{y}$ implies that
$\eeta^1=-\eeta^2$.
Using this relation in \eqref{sd} we obtain
\[
W(\tilde{x}-\tilde{y})-W(x-y)\geq
\langle\eeta^1(x,y),(\tilde{x}-\tilde{y})-(x-y)\rangle\,,
\]
which means that $\eeta^1\in \de W$.
\end{proof}

Applying the main results of the general case we obtain the following

\begin{coro}\label{coro}
 Let assumption \eqref{difference} hold.
If $\eeta$ is a Borel measurable anti-symmetric selection of
$\partial W$, then for any $\mu\in\PP_2(\rd)$,
\begin{equation}\label{sds}
\xxi(x):=\int_{\rd}\eeta(x-y)\,d\mu(y)=\eeta\ast\mu\in
\partial_S\mathcal{W}(\mu).
\end{equation}
Conversely, if $\xxi=\partial^o\w(\mu)$, then there exists a Borel
measurable anti-symmetric selection $\eeta\in\partial W$
such that
\begin{equation*}
\xxi(x)=\int_{\rd}\eeta(x-y)\,d\mu(y)=\eeta\ast\mu.
\end{equation*}
\end{coro}
\begin{proof}
Let $\eeta$ be an antisymmetric selection in $\de W$. In particular
$\eeta(x-y)=-\eeta(y-x)$ for every $x,y\in\rd$ and by Lemma \ref{lem:sd}
$(\eeta^1(x,y),\,\eeta^2(x,y)):=(\eeta(x-y), \eeta(y-x)) \in \de \mathbf{W}(x,y)$.
Applying Theorem \ref{inequality} to the Borel measurable selection
$(\eeta^1,\eeta^2)$ just defined, we get \eqref{sds}.

Conversely, assuming that $\xxi=\partial^o\w(\mu)$, by Theorem
\ref{minimalselectionchar} and Lemma \ref{lem:sd} we obtain that
$$
\xxi(x)=\int_{\rd}\frac12(\eeta^1(x-y)-\eeta^1(y-x))\,d\mu(y) \, .
$$
By choosing $\eeta(z)=\frac12\,(\eeta^1(z)-\eeta^1(-z))$, we
conclude.
\end{proof}

\begin{remark}{\rm
We observe that if $\eeta$ is a Borel measurable selection of
$\partial W$ such that for any $\mu\in\PP_2(\rd)$
\begin{equation*}
\xxi(x):=\int_{\rd}\eeta(x-y)\,d\mu(y)=\eeta\ast\mu\in
\partial_S\mathcal{W}(\mu),
\end{equation*}
then $\eeta$ is antisymmetric. Indeed, let $\eeta\in\partial
W$ be such that
\[
\mathcal{W}(\nu)-\mathcal{W}(\mu)\geq\int_{\rd}\left\langle
\int_{\rd}\eeta(x-y)\,d\mu(y),z-x\right\rangle\,d\gamma(x,z),
\]
for $\gamma\in\Gamma(\mu,\nu)$.
Choosing $\mu=\delta_{x_1}$ and $\nu=\delta_{x_3}$, $\gamma=\mu\times\nu$,
the inequality becomes
\[
0\geq \langle \eeta(0),x_2-x_1\rangle,
\]
and since this must hold for any $x_1,x_2\in\rd$, we deduce
$\eeta(0)=0$.
Moreover, taking into account that $\eeta(0)=0$, if
$\mu=\frac12\delta_{x_1}+\frac12\delta_{x_2}$,
$\nu=\frac12\delta_{x_3}+\frac12 \delta_{x_4}$, with
$|x_1-x_3|\leq |x_1-x_4|$ and $|x_2-x_4|\leq |x_2-x_3|$, the
subdifferential  inequality reduces to
\[
 W(x_3-x_4)-W(x_1-x_2)\geq\langle\eeta(x_1-x_2),x_3-x_1\rangle+\langle\eeta(x_2-x_1),x_4-x_2\rangle.
\]
In particular, for $x_3-x_4=x_1-x_2$ we get
\[
0\geq\langle\eeta(x_1-x_2)+\eeta(x_2-x_1),x_3-x_1\rangle,
\]
which yields $\eeta(x_1-x_2)=-\eeta(x_2-x_1)$ for any $x_1,x_2\in\rd$.}
\end{remark}

\begin{remark}\rm
If $\mu\ll\mathcal{L}^d$ we can conclude that
\[
\int_{\rd}\eeta(x-y)\,d\mu(y)\in \partial\mathcal{W}(\mu)
\]
for any Borel selection $\eeta$ in $\partial W$. Indeed, in this
case the set where $\partial W$ is not a singleton is
$\mu$-negligible. That is, in the integral we can restrict to the
points where $W$ has a gradient (there is no need to select), and
in that case $\nabla W\ast\mu$ belongs to the Wasserstein
subdifferential of $\mathcal{W}$ at $\mu$ (it is actually its
minimal selection), as shown in \cite{CDFLS}.
\end{remark}

\subsection{The radial case} In the radial case, we are able to give a more explicit
characterization of the minimal selection of the Wasserstein subdifferential.
 Before stating our theorem, we recall that
\[
\partial^o W(x)={\mathrm{argmin}}\{|\mathbf{y}|:\mathbf{y}\in\partial W(x)\}.
\]
We have the following
\begin{theorem}\label{radial}
Let $\mathbf{W}$ be convex and such that \eqref{distance} holds.
Then
\begin{equation}\label{minimalradial}
\partial^o\mathcal{W}(\mu)=(\partial^o W)\ast\mu\qquad\forall
\mu\in\PP_2(\mathbb{R}^d).
\end{equation}
\end{theorem}

\begin{proof}
Let $\mu\in\PP_2(\rd)$. Joining together the results of Theorem
\ref{minimalselectionchar} and corollary \ref{coro}, we know that
the $\partial^o\mathcal{W}(\mu)$ has the form of a convolution
with an anti-symmetric selection in the subdifferential of $W$.
Hence, in order to find the explicit form of
$\partial^o\mathcal{W}(\mu)$, we have to minimize the quantity
\[
\|\eeta\ast\mu\|^2_{L^2(\rd,\mu;\rd)}=\int_{\rd}\left|\int_{\rd}\eeta(x-y)\,d\mu(y)\right|^2\,d\mu(x)
\]
among all measurable anti-symmetric selections $\eeta$ in
$\partial W$. Let us introduce some notation. We define, for any
$x\in\rd$, $v\in\rd$, the singleton sets
\[
A^+_v(x):=\{x-v\},\quad A^-_v(x):=\{x+v\}.
\]
Note that by antisymmetry of $\eeta$, we get
\begin{equation}\label{tech}
\int_{A^\pm_v(x)} \eeta(x-y)\,d\mu(y) = \pm\mu\left(\{x\mp
v\}\right) \eeta(v) \, .
\end{equation}
Fix a point $v\in\rd$, different from the origin. We have
\[
\begin{aligned}
\left|\int_{\rd}\eeta(x-y)\,d\mu(y)\right|^2
 &=
\left|\int_{A^+_v(x)\,\cup
A^-_v(x)}\eeta(x-y)\,d\mu(y)+\int_{\rd\setminus(A^+_v(x)\,\cup
A^-_v(x))}\eeta(x-y)\,d\mu(y)\right|^2
\end{aligned}
\]
We suitably expand the square as
$(a+b+c)^2=a^2+2a(b+c)+b^2+2b(a+c)-2ab+c^2$. This way, since
$A_v^+(x)\cup (\rd\setminus(A^+_v(x)\,\cup A^-_v(x)))=\rd\setminus
A_v^-(x)$ and $A_v^-(x)\cup (\rd\setminus(A^+_v(x)\,\cup
A^-_v(x)))=\rd\setminus A_v^+(x)$, the right hand side can be
rewritten as
\[\begin{aligned}
&\left|\int_{A^+_v(x)}\eeta(x-y)\,d\mu(y)\right|^2
 +2\left\langle\int_{A^-_v(x)}\eeta(x-y)\,d\mu(y),\int_{\rd\setminus
A^-_v(x)}\eeta(x-y)\,d\mu(y)\right\rangle
\\+&\left|\int_{A^-_v(x)}\eeta(x-y)\,d\mu(y)\right|^2+
2\left\langle\int_{A^+_v(x)}\eeta(x-y)\,d\mu(y),\int_{\rd\setminus
A^+_v(x)}\eeta(x-y)\,d\mu(y)\right\rangle\\ -&
2\left\langle\int_{A^+_v(x)}\eeta(x-y)\,d\mu(y),
\int_{A^-_v(x)}\eeta(x-y)\,d\mu(y)\right\rangle
 +R(x),
\end{aligned}\]
where the remainder term $R(x)$ is given by
\[
R(x)=\left|\int_{\rd\setminus(A^+_v(x)\,\cup
A^-_v(x))}\eeta(x-y)\,d\mu(y)\right|^2,
\]
so that it does not depend on the values of $\eeta$ at $\pm v$.
Using \eqref{tech}, we are left with
\[\begin{aligned}
\left|\int_{\rd}\eeta(x-y)\,d\mu(y)\right|^2
 &=
\left(\mu(A^+_v(x))+\mu(A^-_v(x))\right)^2|\eeta(v)|^2
\\&\quad+
2\left\langle\int_{A^+_v(x)}\eeta(x-y)\,d\mu(y),\int_{\rd\setminus
A^+_v(x)}\eeta(x-y)\,d\mu(y)\right\rangle
\\&\quad+
2\left\langle\int_{A^-_v(x)}\eeta(x-y)\,d\mu(y),\int_{\rd\setminus
A^-_v(x)}\eeta(x-y)\,d\mu(y)\right\rangle +R(x).
\end{aligned}
\]
Making use of the computed identity, we find
\[
\begin{aligned}
\|\eeta\ast\mu\|^2_{L^2(\rd,\mu;\rd)}&=
\left(\int_{\rd}\left(\mu(A^+_v(x))+\mu(A^-_v(x))\right)^2\,d\mu(x)\right)|\eeta(v)|^2+\mathcal{R}
\\&\quad+
2\int_{\rd}\left\langle\int_{A^+_v(x)}\eeta(x-y)\,d\mu(y),\int_{\rd\setminus
A^+_v(x)}\eeta(x-z)\,d\mu(z)\right\rangle\,d\mu(x)
\\&\quad+
2\int_{\rd}\left\langle\int_{A^-_v(x)}\eeta(x-y)\,d\mu(y),\int_{\rd\setminus
A^-_v(x)}\eeta(x-z)\,d\mu(z)\right\rangle\,d\mu(x)
\\&=
\left(\int_{\rd}\left(\mu(A^+_v(x))+\mu(A^-_v(x))\right)^2\,d\mu(x)\right)|\eeta(v)|^2+\mathcal{R}
\\&\quad+
2\int\!\!\!\int_{\{x-y=v\}}\left\langle\eeta(x-y),\int_{\{z\in\rd:z\neq
x-v\}}\eeta(x-z)\,d\mu(z)\right\rangle\,d\mu(x)\,d\mu(y)
\\&\quad+
2\int\!\!\!\int_{\{x-y=-v\}}\left\langle\eeta(x-y),\int_{\{z\in\rd:z\neq
x+ v\}}\eeta(x-z)\,d\mu(z)\right\rangle\,d\mu(x)\,d\mu(y),
\end{aligned}
\]
where $\mathcal{R}=\int_{\rd}R(x)\,d\mu(x)$. Notice that the
integrals with respect to $z$ can be equivalently taken on the set
$\{z\in\rd:z\neq x,y\}$ in the last two lines. Indeed, in the
first one the equality $x-y=v$ implies $\{z\in\rd:z\neq x-
v\}=\{z\in\rd:z\neq y\}$. Similarly for the second one. Moreover,
we can also neglect the set $\{z=x\}$ since $\eeta(0)=0$ by
anti-symmetry. We exchange $x$ with $y$ in the last integral, so
that
\[
\begin{aligned}&\int\!\!\!\int_{\{x-y=-v\}}\left\langle\eeta(x-y),\int_{\{z\neq
x,y\}}\eeta(x-z)\,d\mu(z)\right\rangle\,d\mu(x)\,d\mu(y)\\&\qquad\qquad=
\int\int_{\{y-x=-v\}}\left\langle\eeta(y-x),\int_{\{z\neq
x,y\}}\eeta(y-z)\,d\mu(z)\right\rangle\,d\mu(y)\,d\mu(x)\\&\qquad\qquad=
\int\int_{\{x-y=v\}}\left\langle-\eeta(x-y),\int_{\{z\neq
x,y\}}\eeta(y-z)\,d\mu(z)\right\rangle\,d\mu(y)\,d\mu(x).
\end{aligned}
\]
Hence from the previous computation we get
\begin{align}
\|\eeta\ast\mu\|^2_{L^2(\rd,\mu;\rd)}&=\left(\int_{\rd}\left(\mu(A^+_v(x))+\mu(A^-_v(x))\right)^2\,d\mu(x)\right)|\eeta(v)|^2+\mathcal{R}
\nonumber\\\label{firstequation} &\quad+
2\int\!\!\!\int_{\{x-y=v\}}\int_{\{z\neq
x,y\}}\left\langle\eeta(v),\eeta(x-z)-\eeta(y-z)\right\rangle\,d\mu(z)\,d\mu(x)\,d\mu(y).
\end{align}
On the set $\{(x,y)\in\rdrd:x-y=v\}$, we rewrite the interior
integral in the following way
\[
\begin{aligned}
\int_{\{z\neq
x,y\}}\left\langle\eeta(v),\eeta(x-z)-\eeta(y-z)\right\rangle\,d\mu(z)&=
\int_{\{z\neq
x,y,x+v,y-v\}}\left\langle\eeta(v),\eeta(x-z)-\eeta(y-z)\right\rangle\,d\mu(z)
\\&\quad+
\int_{\{z=x+v\}\cup\{z=y-v\}}\langle\eeta(v),\eeta(2v)-\eeta(v)\rangle\,d\mu(z)
\\&=
\int_{\{z\neq
x,y,x+v,y-v\}}\left\langle\eeta(v),\eeta(x-z)-\eeta(y-z)\right\rangle\,d\mu(z)
\\&\quad+
\left(\mu(A^-_v(x))+\mu(A^+_v(y))\right)\left(\langle\eeta(v),\eeta(2v)\rangle-|\eeta(v)|^2\right),
\end{aligned}
\]
and thus
\[
\begin{aligned}
&\int\!\!\!\int_{\{x-y=v\}}\int_{\{z\neq
x,y\}}\left\langle\eeta(v),\eeta(x-z)-\eeta(y-z)\right\rangle\,d\mu(z)\,d\mu(x)\,d\mu(y)
\\&\qquad\qquad=
\int\!\!\!\int_{\{x-y=v\}}\int_{\{z\neq
x,y,x+v,y-v\}}\left\langle\eeta(v),\eeta(x-z)-\eeta(y-z)\right\rangle\,d\mu(z)\,d\mu(x)\,d\mu(y)
\\&\qquad\qquad\quad+
2\left(\int_{\rd}(\mu(A^+_v(x))\mu(A^-_v(x)))\,d\mu(x)\right)\left(\langle\eeta(v),\eeta(2v)\rangle-|\eeta(v)|^2\right).
\end{aligned}
\]
Here, in the last term we used
$$
\int\!\!\!\int_{\{x-y=v\}} \!\!\!\!\!\!\mu(A^-_v(x))
\,d\mu(x)\,d\mu(y) = \int\!\!\!\int_{\{x-y=v\}}
\!\!\!\!\!\!\mu(A^+_v(y)) \,d\mu(x)\,d\mu(y) = \int
\mu(A^+_v(x))\mu(A^-_v(x)) \,d\mu(x)\,.
$$
Substituting in \eqref{firstequation} we get the final expression
\begin{align}
\|\eeta\ast\mu\|^2_{L^2(\rd,\mu;\rd)}&=\left(\int_{\rd}\left(\mu(A^+_v(x))-\mu(A^-_v(x))\right)^2\,d\mu(x)\right)|\eeta(v)|^2+\mathcal{R}
\nonumber \\\nonumber &\quad+
2\int\!\!\!\int_{\{x-y=v\}}\int_{\{z\neq
x,y,x+v,y-v\}}\!\!\!\!\!\!\left\langle\eeta(v),\eeta(x-z)-\eeta(y-z)\right\rangle\,d\mu(z)\,d\mu(x)\,d\mu(y)
\\\label{secondequation}&\quad+
4\left(\int_{\rd}(\mu(A^+_v(x))\mu(A^-_v(x)))\,d\mu(x)\right)\langle\eeta(v),\eeta(2v)\rangle.
\end{align}

Since $W$ is radial, for all $v\in\rd$, $\eeta(v)$ takes the form
\begin{equation}\label{radialpositive}
\eeta(v)=\frac{\eta(|v|)}{|v|}\,v,
\end{equation}
where $\eta$ denotes the generic selection in the subdifferential
of the profile function $w$ in \eqref{distance}.
Convexity implies that $\eta(|v|)\geq 0$, and thus
\begin{equation}\label{radialpositive2}
\langle\eeta(v),v\rangle= \left|\eeta(v)\right|\left|v\right|
\end{equation}
for all $v\in\rd$. The convexity hypothesis gives
\[
\langle x-y,\eeta(x-z)-\eeta(y-z)\rangle\geq0,
\]
for any $x,y,z$. Combining this with \eqref{radialpositive} since
$\eta(|v|)\geq 0$, we conclude that
\[
\langle \eeta(x-y),\eeta(x-z)-\eeta(y-z)\rangle\geq 0.
\]
From this last inequality it is readily seen that the scalar
products appearing in \eqref{secondequation} are (pointwise)
nonnegative. Therefore, since the remainder $\mathcal{R}$ is
independent on $\eeta(v)$ and $\eeta(-v)$, it is clear from
\eqref{secondequation} and \eqref{radialpositive2} that
$\|\eeta\ast\mu\|^2_{L^2(\rd,\mu;\rd)}$ at least does not increase
if we choose the value of $\eeta$ at $v$ to be the one that
minimizes $|\eeta(v)|$. By the arbitrariness of $v\in\rd$ we get
\eqref{minimalradial}.
\end{proof}

\begin{example}\rm
The above result fails if we omit the convexity assumption.
Indeed, let us consider a $1$-dimensional example. Let
\[
\widehat{W}(x)=\frac12\,|x^2-1|.
\]
Notice that this function is radial and $-1$-convex, and its
subdifferential is
\[
\partial W(x)=\left\{
\begin{aligned}
x\quad &\mbox{for $|x|>1$},\\
-x\quad &\mbox{for $|x|<1$},\\
[-1,1]\quad&\mbox{for $x=\pm 1$}.
\end{aligned}
\right.
\]
 Let us consider the measure
$ \mu=\frac13 \delta_{x_1}+\frac13 \delta_{x_2}+\frac13
\delta_{x_3}$. We have to minimize the quantity
\[
\|\eeta\ast\mu\|_{L^2(\mathbb{R},\mu;\mathbb{R})}=\frac{1}{27}\sum_{j=1}^3\left|\sum_{i=1}^3\eeta(x_j-x_i)\right|^2
\]
among all measurable antisymmetric selections $\eeta$ in $\partial
W$.

If we let $x_1=1$, $x_2=0$, $x_3=3/4$, the only points where it is
needed to select are $\pm 1$, corresponding to $\pm (x_1-x_2)$,
hence expanding the sum above (using the antisymmetry) it is clear
that we reduce to find the minimizer of
\[
\min\{\eeta(x_1-x_2)^2+\eeta(x_1-x_2)(\eeta(x_1-x_3)-\eeta(x_2-x_3)):\eeta(x_1-x_2)\in[-1,1]\}.
\]
Here $x_1-x_3=1/4$ (so that $\eeta(x_1-x_3)=-1/4$) and
$x_2-x_3=-3/4$ (so that $\eeta(x_2-x_3)=3/4$). Then, letting
$y=\eeta(x_1-x_2)$, we are left with the problem
$\min\{y^2-y:y\in[-1,1]\}$, whose solution is $y=1/2$. This is
different from the element of minimal norm in $\partial
W(x_2-x_1)$, which of course is $0$.

We also point out that in this non convex case the choice of the
selection is not independent from the measure $\mu$. Indeed, if we
change the value of $x_3$ to be, for instance, $-1/4$, we have
$\eeta(x_1-x_3)=5/4$ and $\eeta(x_2-x_3)=-1/4$, then we have to
solve $\min\{y^2+3y/2:y\in[-1,1]\}$, and the solution is
$y=\eeta(1)=-\eeta(-1)=-3/4$.
\end{example}

\begin{example}\upshape{
The result of Theorem \ref{radial} fails if we omit the radial
hypothesis on $W$. As a counterexample we provide a convex
function $\mathbf{W}$ satisfying all the assumptions
\eqref{sym},\eqref{convex}, \eqref{growth}, \eqref{difference} and
a measure $\mu\in\PP_2(\mathbb{R}^2)$ such that
\[
\partial^o \mathcal{W}(\mu)\neq (\partial^o W)\ast\mu.
\]

Let the graph of $W$ be a pyramid
with vertex in the origin, with varying slopes, given by
\[
\nabla W(x,y)=\left\{\begin{array}{rl}\vspace{6pt} (1,0)\quad
&\mbox{for}\quad  0<x<1,\; -x<y<x,\\
(\theta,0)\quad &\mbox{for}\quad x>1,\; -x<y<x,
\end{array}\right.
\]
where $\theta>2$. $W$ is then defined by symmetry, such that its
level sets are squares centered in the origin. Let $\mu=\frac13
\delta_{x_1}+\frac13 \delta_{x_2}+\frac13\delta_{x_3}$, where
\[
x_2-x_1=(1,1),\quad x_3-x_2=(-1/2-\eps,1/2),\quad
x_3-x_1=(1/2-\eps,3/2).
\]
Among these points, for small enough $\eps>0$, $\partial W$ is not
a singleton only at $x_2-x_1$, and in particular it is the convex
set $K$ of $\mathbb{R}^2$ defined as
\[
K=\{(x,y)\in\mathbb{R}^2:\; x\geq 0,\;y\geq 0,\; 1-x\leq y\leq
\theta-x\}.
\]
We let $\eeta_{ij}$, $i,j=1,2,3$ denote the generic element of
$\partial W$ at $x_i-x_j$. In this particular case
\[
\partial^o\mathcal{W}(\mu)={\mathrm{argmin}}\left\{\frac1{27}\,\sum_{j=1}^3\left|\sum_{i=1}^3
\eeta_{ji}\right|^2: \eeta\in\partial W,
\eeta_{ji}=-\eeta_{ij}\right\}.
\]
Since $\eeta_{13}=\nabla W(x_1-x_3), \eeta_{23}=\nabla
W(x_2-x_3)$, we are left to minimize with respect to the unique
variable $\eeta_{21}$, that is, the minimization problem above
reduces to
\[
\min_{\eeta_{21}\in K}
|-\eeta_{21}+\eeta_{13}|^2+\frac{1}{27}\,|\eeta_{23}+\eeta_{21}|^2+\frac{1}{27}|\eeta_{32}+\eeta_{31}|^2.
\]
We have $\eeta_{13}=(0,-\theta)$ and $\eeta_{23}=(1,0)$, and hence
it is immediate to check the solution is the minimizer of
\begin{equation}\label{quadratic}
|\eeta_{21}|^2+\langle\eeta_{21},(1,\theta)\rangle.
\end{equation}
If $\eeta=\partial^o W$ we have that $\eeta_{21}$ is the element
of minimal norm in $K$, that is $(1/2,1/2)$, and in this case the
quantity above takes the value $1+\theta/2$. But $\theta>2$, so
that the minimum value in $K$ of the quadratic expression
\eqref{quadratic} is $2$, attained in a different point, that is
$\eeta_{21}=(1,0)$. }
\end{example}


\section{Particle system}\label{particlesection}

As in \cite{CDFLS}, the well-posedness result in Theorem
\ref{gradientflow} of measure solutions allows to put in the same
framework particle and continuum solutions. Assume that we are
given $N$ pointwise particles, each carrying a mass $m_i$, with
$\sum_i m_i=1$. Let $x_i(t)$ be the position in $\rd$ of the
$i$-th particle at time $t$. Let $\eeta^1(x,y):\rdrd\to\rd$ and
$\eeta^2(x,y):\rdrd\to\rd$ be selections in the subdifferentials
$\partial_1\mathbf{W}$ and $\partial_2\mathbf{W}$ respectively.
 We consider the system
\[
\frac{dx_i}{dt}=\frac12\sum_{j=1}^N m_j
(\eeta^1(x_j-x_i)+\eeta^2(x_i-x_j)),\qquad i=1,\ldots, N.
\]
If the $x_i$ are absolutely continuous curves, the empirical
measure $\mu(t)=\sum_{i=1}^N m_i\delta_{x_i(t)}$ solves the PDE
\[
\frac{d}{dt}
\mu(t)-\frac12\,\div\left(\left(\int_{\rd}\eeta^1(\cdot,y)+\eeta^2(y,\cdot)\,d\mu(y)\right)\mu(t)\right)=0.
\]
Among these selections, we have the minimal one in $\partial
\mathcal{W}(\mu)$, hence we have correspondence between our
equation \eqref{particular} and the ODE system above. In
particular, if we are in the framework of assumption
\eqref{difference}, the velocity vector field of the continuity
equation is written as a convolution and the corresponding ODEs
takes the form
\begin{equation}\label{particle}
\frac {dx_i}{dt}=\sum_{j=1}^N m_j\eeta_{ji},\qquad i=1,\ldots, N,
\end{equation}
where $\eeta_{ij}:=\eeta(x_i-x_j)$ and $\eeta$ is the suitable
measurable anti-symmetric selection in $\partial W$.

\begin{remark}[\textbf{Characterization of the element of minimal norm in the radial convex case for particles}]\rm
In the particular case of a system of particles, it is more
immediate to see how the proof of Theorem \ref{radial} works.
Actually, it was the origin of the idea in how to get the proof in
previous section. We include it since we believe it is an
instructive proof. Suppose that $W$ is convex, so that the minimal
selection in $\partial \mathcal{W}(\mu)$ is expected, after
Theorem \ref{radial}, to be $\partial^o W\ast\mu$. In fact, we
have in this case
\begin{equation*}
\mu=\sum_{i=1}^N m_i\delta_{x_i}(x),\qquad m_1+\ldots+m_N=1,
\end{equation*}
where $x_i\in \rd$. After Theorem \ref{minimalselectionchar}, we
know that we have to search for the  Borel antisymmetric selection
$\eeta$ in the subdifferential of $W$ that minimizes the
$L^2(\rd,\mu;\rd)$ norm of $\eeta\ast\mu$. We have
\[
\|\eeta\ast\mu\|^2_{L^2(\rd,\mu;\rd)}=\sum_{j=1}^N
m_j\left|\sum_{i=1}^N m_i \eeta(x_j,x_i)\right|^2.
\]
For simplicity we are letting again $\eeta(x_j,x_i)=\eeta_{ji}$.
Hence we have to solve the problem
\begin{equation*}
\min\left\{\sum_{j=1}^N m_j\left|\sum_{i=1}^N m_i
\eeta_{ji}\right|^2: \eeta\in\partial W,
\eeta_{ji}=-\eeta_{ij}\right\}.
\end{equation*}
Suppose first that $W$ is differentiable in all the points
$x_j-x_i$. In this case we have $\eeta_{ji}=\nabla W(x_j-x_i)$,
and the quantity to minimize simply takes the (single) value
\[
\sum_{j=1}^N m_j\left|\sum_{i=1}^N m_i \nabla W(x_j-x_i)\right|^2,
\]
which is then the square norm of the minimal selection in
$\partial\mathcal{W}$. Otherwise, taking advantage of the
anti-symmetry of each selection, let us rearrange the terms
highlighting  $\eeta_{JI}$, where $I$ and $J$ are fixed nonequal
numbers between $1$ and $N$, and let $x_J-x_I=x_0$. In order to
keep simple, we make the assumption that no other couple $(i,j)$
is such that $x_j-x_i=x_0$. This way, it is not difficult to see
that, expanding the squares, we can write
\begin{equation}\label{rearrangement}
\sum_{j=1}^N m_j\left|\sum_{i=1}^N m_i \eeta_{ji}\right|^2=m_I
m_J(m_I+m_J) |\eeta_{JI}|^2 +2m_I m_J\sum_{{k=1}\atop{k\neq
J,\,I}}^N
m_k\left\langle\eeta_{JI},(\eeta_{Jk}-\eeta_{Ik})\right\rangle +R,
\end{equation}
where $R$ does not depend on the value of $\eeta$ at $\pm x_0$. We
claim that each scalar product in the sum above is nonnegative,
since by convexity of $W$, any selection $\eeta_{ij}$ in $\partial
W(x_i-x_j)$ is monotone in the sense that
\[
\langle
\eeta_{jk}-\eeta_{ik},x_j-x_k-x_i+x_k\rangle=\langle\eeta_{jk}-\eeta_{ik},x_j-x_i\rangle\geq
0.
\]
By assumption {\it v)}, W is radial ($W(\cdot)=w(|\cdot|)$),
and if $\eeta_{ji}$ is a selection in $\partial w
(|x_j-x_i|)$, then $\eeta_{ji}$ is positive by convexity of $W$ and
\[
\left\langle\eeta_{ji},\eeta_{jk}-\eeta_{ik}\right\rangle=\frac{\eeta_{ji}}{|x_j-x_i|}\,\langle
x_j-x_i,\eeta_{jk}-\eeta_{ik}\rangle\geq 0\quad \forall i,j,k.
\]
This shows that the claim is correct. Then if we minimize in
\eqref{rearrangement} only with respect to the admissible values
of $\eeta_{JI}$, we see that the minimum argument is found if we
take $\eeta_{JI}=\partial^o W(x_J-x_I)$ (and the same for
$\eeta_{IJ}$, by anti-symmetry). But by construction, such minimum
argument does not depend on the value of $\eeta$ in the other
points.
We conclude that if $\eeta$ is a solution of the full minimization
problem, then $\eeta_{JI}$ is forced to be the element of minimal
norm in $\partial W(x_J-x_I)$. By the arbitrariness of $I,J$ we
conclude: the solution is $\partial^o W\ast\mu$.
\end{remark}

\begin{remark}[\textbf{Collapse}]\rm
In the radial-convex case, the finite-time collapse argument in
\cite{CDFLS} carries over. Indeed, it is enough to substitute
$\nabla W$ therein with the new object $\partial^o W$. Consider
the equation \eqref{particle} for a system of $N$ particles. We
notice that, since $\partial^o W$ is anti-symmetric, the center of
mass $\frac1N\sum_{j=1^N}x_j$ does not move during the evolution.
We denote such point by $x_0$. Then, assuming for simplicity that
the weights are uniform, there holds
\[
\frac{d(x_i-x_0)}{dt}=-\frac1N \sum_{j=1}^N\partial^o
w(|x_i-x_j|)\frac{x_i-x_j}{|x_i-x_j|}.
\]
Multiplying by $x_i-x_0$ we find
\begin{equation}\label{particleradial} \frac12\frac
d{dt}|x_i-x_0|^2=-\frac1N \sum_{j=1}^N\partial^o
w(|x_i-x_j|)\frac{\langle x_i-x_j,x_i-x_0\rangle}{|x_i-x_j|}.
\end{equation}
Now, let $R(t):=\max_{i\in\{1,\ldots,N\}}\{|x_i(t)-x_0|\}$.  If
$I\in\{1,\ldots,N\} $ is such that $|x_I(t)-x_0|=R(t)$,  at that
time we clearly have $\langle x_I-x_j,x_I-x_0\rangle\geq 0$. For
any time and for any $I$ with this property, from
\eqref{particleradial} we see that the derivative of $|x_I-x_0|^2$
is nonpositive, so that its maximum governs the evolution of $R$:
\[\begin{aligned}
\frac12\frac{d}{dt} R^2(t)&=\max_{\{i:x_i(t)=R(t)\}}-\frac1N
\sum_{j=1}^N\partial^o
w(|x_i-x_j|)\frac{\langle x_i-x_j,x_i-x_0\rangle}{|x_i-x_j|}\\
&\leq \max_{\{i:x_i(t)=R(t)\}}-\frac1N
\frac{\Omega}{2R(t)}\sum_{j=1}^N
{\langle x_i-x_j,x_i-x_0\rangle}\\
&\leq \max_{\{i:x_i(t)=R(t)\}}-\frac1N
\frac{\Omega}{2R(t)}\,|x_i-x_0|^2=-\frac{\Omega}{2N} R(t),
\end{aligned}
\]
where $\Omega=\inf_{(0,+\infty)}\partial^o w$. Here we used
the elementary bound $|x_i(t)-x_j(t)|\leq 2R(t)$ and the fact that
$\sum_{j=1}^N\langle x_I-x_j,x_I\rangle=|x_I-x_0|^2$ if
$I\in\{1,\ldots,N\}$ is such that $|x_I(t)|=R(t)$. The solution
then reaches the asymptotic state $\mu_\infty=\delta_{x_0}$ in
finite time if $\Omega>0$. We can also generalize the result to
the case $\Omega=0$ under suitable assumptions on the behavior of
$\omega$ at the origin such that a Gronwall estimate can be
achieved. For instance if we ask $\frac{\partial^o w(x)}{x}$
to be decreasing on some interval $(0,\eps)$ we are done. For all
the details and a more general discussion on this issue we refer
to \cite{CDFLS}. We only remark that if we have finite time
collapse in the center of mass for a system of particles, by the
general stability properties of gradient flow solutions  we can
also deduce the finite time collapse for general compactly
supported initial data.
\end{remark}


\section{Examples of asymptotic states in the non convex
case}\label{asymptoticsection} In this section, suppose that
assumptions \eqref{difference} and \eqref{distance} hold. We have
already seen in the case of particle collapse that, for convex
nonnegative potentials, there is the asymptotic state consisting
of a single Dirac mass. It is easily seen that any Dirac mass
minimizes functional $\mathcal{W}$. So, in this case, by the
translation invariance property, the only asymptotic solution is
in fact the Dirac mass in the center of mass of the initial datum.

In the repulsive-attractive case, one can expect a much bigger
variety of stationary solutions. This is the case as proved in 1D
in \cite{FellnerRaoul1,FellnerRaoul2} in which the authors show
that the set of stationary states for short-range repulsive
long-range attractive potentials can be very large and
complicated. Moreover, they give examples in which the stationary
states are composed of a finite number of Dirac deltas at points
and others in which one has integrable compactly supported
stationary solutions. They also show that to have integrable or
not stationary states depends in how strong the repulsion is at
the origin. Some numerical computations indicates that this is
also the case in more dimensions \cite{KSUB}. They show that the
set of stable stationary states can be large and with complicated
supports arising from instability modes of the uniform
distribution on a sphere. Finally, in \cite{BCLR} we have studied
the stability for radial perturbations of the uniform distribution
on a suitable sphere for general radial repulsive-attractive
potentials. We will show a related example below.

In this section we are going to present some examples, based on
the available characterization of the velocity vector field of the
continuity equation of properties of the set of stationary
solutions. In particular, we will consider the $(-1)$-convex
potentials in one dimension given by
\begin{equation*}
\widetilde{W}(x)=\frac12\,|x^2-1|^2,\qquad
\widehat{W}(x)=\frac12\,|x^2-1|.
\end{equation*}
Both cases correspond to attractive-repulsive potentials with the
same behavior in the repulsive part at the origin. However, the
change from repulsive to attractive in one case is smooth and in
the other, it is only Lipschitz. In fact, in the two cases there
many analogies, but also some different behaviors, as we are going
to show with the next propositions. First of all, we search for
stationary states made by a finite number of particles.

\begin{pro}
There exist two-particles stationary states for $\tw$ and $\hw$.
\end{pro}
\begin{proof}
Let $\mu=m_1\delta_{x_1}+m_2\delta_{x_2}$, with $m_1+m_2=1$. For
functional $\tw$ we have to impose the condition
\begin{equation}\label{2equations}
\frac {dx_i}{dt}=m_j\nabla\tw(x_j-x_i)=0,
\end{equation}
for $i,j=1,2, i\neq j$. But $\nabla \tw(x)=2x(x^2-1)$. Hence it is
enough to choose $x_1,x_2$ such that $|x_1-x_2|=1$, independently
of the weights. We see that we have infinitely many stationary
states made by two Dirac deltas. In the case of $\hw$, the
subdifferential is
\begin{equation}\label{subgradient}
\partial\hw(x)=\left\{\begin{aligned}x\quad&\mbox{if
$|x|>1$},\\
-x\quad&\mbox{if $|x|<1$},\\
[-1,1]\quad&\mbox{if $|x|=1$}.
\end{aligned}\right.
\end{equation}
In order to find a stationary state, we have to solve
\eqref{2equations} with $\eeta$ in place of $\nabla\widetilde{W}$,
where  $\eeta$ is the suitable anti-symmetric selection in
$\partial\widehat{W}$ realizing the minimal norm in
$\partial\mathcal{W}(\mu)$ (see Corollary \ref{coro}). Such
selection is then found minimizing the quantity
\[
\begin{aligned}
\|\eeta\ast\mu\|^2_{L^2(\mathbb{R},\mu;\mathbb{R})}=\sum_{j=1}^2
m_j \left|\sum_{i=1}^2 m_i\eeta(x_j-x_1)\right|^2
&=m_1m_2(m_1+m_2)(\eeta(x_1-x_2))^2
\end{aligned}
\]
among the admissible selections $\eeta$ (we are using the
anti-symmetry). Let again $|x_2-x_1|=1$, so that it is clear that
the minimum above is zero, attained for $\eeta(1)=-\eeta(-1)=0$.
And this way, the two equations \eqref{2equations} are still
satisfied.
\end{proof}

\begin{pro}
For both functionals $\tw$ and $\hw$,  there are no absolutely
continuous stationary states in one dimension.
\end{pro}
\begin{proof}
Let us consider functional $\hw$. The argument is based on the
fact that, if $\mu$ is absolutely continuous, it does not charge
the points of non-differentiability of $\hw$. For a measure $\mu$
to be stationary, we have to verify that the corresponding
velocity vector field vanishes. That is $ \eeta\ast\mu=0, $ where
$\eeta$ is the usual optimal selection in $\partial \hw$, as in
Corollary \ref{coro}. Suppose that $\mu$ is a stationary state and
that $\mu=\rho\mathcal{L}^1$, for some $\rho\in L^1(\mathbb{R})$,
then
\[
\int_{\{|x-y|>1\}}(x-y)\rho(y)\,dy-\int_{\{|x-y|<
1\}}(x-y)\rho(y)\,dy=0.
\]
By the translation invariance property, we can fix without loss of
generality the center of mass, so we let $\int_{\mathbb{R}}
y\rho(y)\,dy=0$. We deduce
\[
2x\int_{\{|x-y|>1\}}\rho(y)\,dy-x=2\int_{\{|x-y|>1\}}y\rho(y)\,dy,
\]
hence
\[
2x\left(\int_{-\infty}^{x-1}\rho(y)\,dy+\int_{x+1}^{+\infty}\rho(y)\,dy\right)-x=2\left(\int_{-\infty}^{x-1}y\rho(y)\,dy
+\int_{x+1}^{+\infty}y\rho(y)\,dy\right).
\]
Let us denote the term in the parenthesis in the left side by
$\Theta(x)$ and let us take the derivative with respect to $x$. We
have
\[
2\Theta(x)+
2x(\rho(x-1)-\rho(x+1))-1=2((x-1)\rho(x-1)-(x+1)\rho(x+1)),
\]
which yields $2\Theta(x)-1=-2(\rho(x-1)-\rho(x+1))$, that is
$\Theta^\prime(x)=-\Theta(x)+\frac12$. We find
\[
\Theta(x)=ke^{-x}+\frac12,\quad k\in\mathbb{R},
\]
then  $\rho(x-1)-\rho(x+1)=-ke^{-x}$. But the integral of $\rho$
is $1$, so $k=0$ and we are left with $\rho(x-1)=\rho(x+1)$. This
is a contradiction, since $\rho$ can not be periodic in this case.
The proof for $\tw$ is analogous, we omit the details.
\end{proof}

The following are more examples of stationary states

\begin{example}\rm
There are stationary states for functional $\hw$ of the form
\[
\mu=m_1\delta_{x_1}+m_2\delta_{x_2}+m_3\delta{x_3},
\]
with $m_1+m_2+m_3=1$.
 Indeed, we have to verify that
\begin{equation}\label{three}\left\{\begin{aligned} \frac
{dx_1}{dt}&=m_2\eeta(x_2-x_1)+m_3\eeta(x_3-x_1)=0\\
\frac{dx_2}{dt}&=m_1\eeta(x_1-x_2)+m_3\eeta(x_3-x_2)=0\\
\frac{dx_3}{dt}&=m_1\eeta(x_1-x_3)+m_2\eeta(x_2-x_3)=0,
\end{aligned}\right.\end{equation}
where, as usual, $\eeta$ represents the anti-symmetric selection
in the subdifferential \eqref{subgradient} given by Corollary
\ref{coro}.

 For instance, let us search for a solution in the
following range
\begin{equation}\label{range}
x_2-x_1=1,\quad x_3-x_1>1,\quad 0<x_3-x_2<1. \end{equation} We
begin searching for the right selection. We use the notation
$\eeta_{ij}:=\eeta(x_i-x_j)$. Clearly we need to select only at
$\pm(x_2-x_1)$. Taking the anti-symmetry into account, and
recalling that the subdifferential of $\widehat{W}$ is
\eqref{subgradient} and that the relations \eqref{range} hold,
there is
\[
\begin{aligned}
\|\eeta\ast\mu\|^2_{L^2(\mathbb{R},\mu;\mathbb{R})}&=\sum_{j=1}^3
m_j \left|\sum_{i=1}^3 m_i\eeta(x_j-x_1)\right|^2\\
&=m_1(m_2^2\eeta_{12}^2+2m_2m_3(x_1-x_3)\eeta_{12})+m_2(m_1^2\eeta_{21}^2+2m_1m_3(x_3-x_2)\eeta_{21})+R_1\\
&=m_1m_2(m_1+m_2)\eeta_{12}^2+2m_1m_2m_3(x_1+x_2-2x_3)\eeta_{12}+R_2,
\end{aligned}
\]
where the remainders $R_1,R_2$ do not depend on the value of
$\eeta$ at $\pm1$, hence we only have to minimize with respect to
$\eeta_{12}$ on the interval $[-1,1]$. We have a quadratic
function, so that if the vertex of the parabola is  on the right
of the interval $[-1,1]$, then the minimizer is found for
$\eeta_{12}=1$, hence we get $\eeta(-1)=1$ and then, by
anti-symmetry, $\eeta(1)=-1$. Computing the vertex position, this
condition is
\begin{equation}\label{vertex}
    \frac{m_3(x_3-x_2)+m_3(x_3-x_1)}{m_1+m_2}\,\geq1.
\end{equation}
Hence, if such condition holds, making use of \eqref{subgradient}
the first two equations in \eqref{three} reduce to
\[
\left\{\begin{aligned}&-m_2+m_3(x_3-x_1)=0,\\
&m_1-m_3(x_3-x_2)=0.\end{aligned}\right.
\]
We deduce
\[
x_3-x_1=\frac{m_2}{m_3}\quad\mbox{and}\quad
x_3-x_2=\frac{m_1}{m_3},
\]
and since $1=x_2-x_1=(x_3-x_1)-(x_3-x_2)$, we have
\[
m_3=m_2-m_1.
\]
We may pose, for $0<\alpha<\frac14$,
\[
m_1=\frac14 -\alpha,\quad m_2=\frac12,\quad m_3=\frac14 +\alpha.
\]
Then we find
\begin{equation}\label{differences}
x_3-x_1=\frac2{1+4\alpha},\quad
x_3-x_2=\frac{1-4\alpha}{1+4\alpha}.
\end{equation}
This way, the constraints in \eqref{range} are satisfied.
Moreover, it is immediate to check that \eqref{three} is solved
and that

\[
\frac{m_3(x_3-x_2)+m_3(x_3-x_1)}{m_1+m_2}\,=1,\] so that
\eqref{vertex} is verified and the computation is indeed
consistent.

 In order to find
the three points, we can for instance fix the center of mass in
the origin:
\[
\frac{1-4\alpha}{4}\,x_1+\frac12\,x_2+\frac{1+4\alpha}{4}\,x_3=0.
\]
Together with \eqref{differences}, this gives
\[
x_1=-1,\quad x_2=0,\quad x_3=\frac{1-4\alpha}{1+4\alpha}.
\]
We conclude that, for $0<\alpha<\frac14$,
\[
\mu=\left(\frac14
-\alpha\right)\delta_{-1}+\frac12\,\delta_0+\left(\frac14
+\alpha\right)\delta_{\frac{1-4\alpha}{1+4\alpha}}
\]
is a stationary state.
\end{example}

\begin{remark}\rm
In the case of $\widetilde{W}$, it is very easy to construct an
analogous example, solving system \eqref{three}, where this time
the actual gradient of $\widetilde{W}$ appears. In both cases, it
seems clear that the procedure can be repeated for finding
infinitely many stationary states with $N$ Dirac masses for any
$N>3$.
\end{remark}

We conclude with an example in two space dimensions. The reference
functional is simply the radial of $\hw$, still denoted by $\hw$,
that is
\[
\hw(x)=\frac12 ||x|^2-1|,\quad x\in\mathbb{R}^2.
\]
The functional is still $(-1)$-convex, and in this case
\[
\partial\widehat{W}(x)=\left\{\begin{aligned}x\quad&\mbox{ if $|x|>1$}\\
-x\quad &\mbox{ if $|x|\leq 1$}\\
[-1,1]{x}\quad&\mbox{ if $|x|=1$.}
\end{aligned}\right.
\]

\begin{example}\label{circle}\rm
Let $\sigma_R$ denote the uniform measure on the circumference
$\partial B_R(0)$, of radius $R$, centered in the origin. There
exists $R>0$ such that the measure $\sigma_{R}$ is a stationary
state for functional $\hw$.

Indeed, we can show that for any $x\in\partial B_R(0)$ there holds
$\eeta\ast\sigma_R=0$ for a suitable choice of the radius $R$,
$\eeta$ being the optimal selection of Corollary \ref{coro}.
Explicitly, the convolution is
\[
\int_{\{|x-y|>1\}}(x-y)\,d\sigma_R(y)-\int_{\{|x-y|\leq
1\}}(x-y)\,d\sigma_R(y),
\]
and the set of points $\{y:|x-y|=1\}$, where one should select, is
negligible. Fix $x$ on the circle. We let
$(\mathbf{e}_1,\mathbf{e}_2)$ be an orthogonal base in
$\mathbb{R}^2$, where $\mathbf{e}_1$ is the direction of $x$, so
that $x=R\mathbf{e}_1$. Hence we have to solve
\[
R\mathbf{e}_1\sigma_R(\{|x-y|\!>1\})-R\mathbf{e}_1\sigma_R(\{|x-y|\!\leq1\})
-\int_{\{|x-y|>1\}}y\,d\sigma_R(y)+\int_{\{|x-y|\leq
1\}}y\,d\sigma_R(y)=0.
\]
We write the integrals in polar coordinates with respect to the
origin and the vector $\mathbf{e}_1$. In this system, we let
$\alpha=\alpha(R)$ denote the (positive) angle corresponding to
the intersection point between $\partial B_R(0)$ and the circle of
radius $1$ centered in $x$ (see Figure 1 below). In particular
\begin{equation}\label{senalfa}
\sin\alpha(R)=\frac{\sqrt{R^2-\frac14}}{R^2}.
\end{equation}
Since $y=R\cos\theta\mathbf{e}_1+R\sin\theta\mathbf{e}_2$, it is
immediately seen, by oddness of the sine function, that the
equation in the direction of $\mathbf{e}_2$ is identically
satisfied. In the direction of $\mathbf{e}_1$ we find
\[
2R\int_{\alpha}^\pi R\,d\theta-2R\int_0^\alpha
R\,d\theta-2R\int_\alpha^\pi\cos\theta
R\,d\theta+2R\int_0^\alpha\cos\theta R\,d\theta=0.
\]
Hence we have to solve
\begin{equation*}
f(R):=\pi-2\alpha(R)+2\sin\alpha(R)=0.
\end{equation*}
If $R<\frac12$, the computation does not make sense, but indeed we
can not have a stationary state for $R<\frac12$, since in this
case any point in $\partial B_R(0)$ has distance lower than $1$
from $x$, so that for each of them the effect on $x$ is a
repulsion, and $x$ tend to move far from the origin. Taking
\eqref{senalfa} into account, if $R=\frac12$, we have
$\alpha(R)=\pi$, hence the value of  $f$ at $\frac12$ is $-\pi$.
As $R$ increases from $\frac12$ to $+\infty$, the angle
$\alpha(R)$ decreases from $\pi$ to $0$. Notice that the function
\[
R\mapsto\frac{\sqrt{R^2-\frac14}}{R^2}
\]
is increasing from $\frac12$ to $\frac{\sqrt{2}}2$, where it has
its maximum, and is decreasing in $(\frac{\sqrt{2}}2,+\infty)$. On
the other hand, $\sin(x)-x$ is a decreasing function. Since
$f(\sqrt{2}/2)=2$, we conclude that $f$ has only one zero, found
in the interval $(1/2,\sqrt{2}/2)$. If $R_0$ is the zero, for
$R=R_0$ the measure $\sigma_R$ is stationary.
\end{example}

\begin{figure}[h]
\includegraphics[width=16cm]{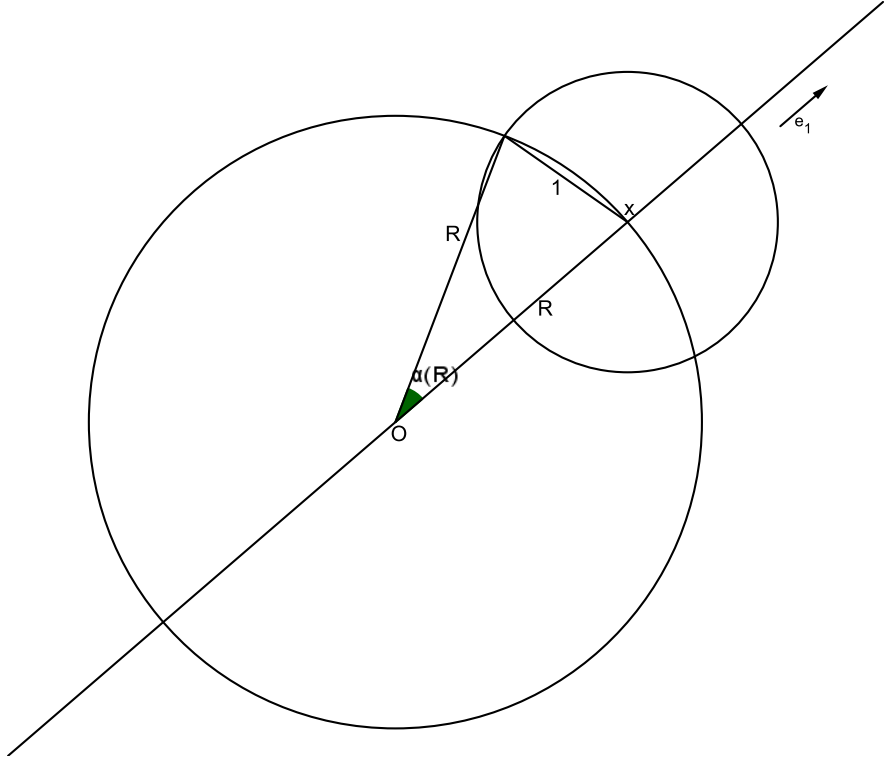}
\caption{The construction of Example \ref{circle}}
\end{figure}

\section{Appendix: vector and plan subdifferential}\label{appendix}
Here we give a more complete overview about the Wasserstein
subdifferential. In \cite[¤10.3]{AGS}, the theory is developed for
functionals $\Phi:\PP_2(\rd)\to(-\infty,+\infty]$ such that
\begin{equation}\label{hp1}
\Phi:\PP_2(\rd)\to(-\infty,+\infty]\;\mbox{ is proper and lower
semicontinuous in $\PP_2(\rd)$ }
\end{equation}
and
\begin{equation}\label{hp2}
    \Phi(\cdot)+\frac1{2\tau}\,d_W^2(\cdot,\mu)\;\mbox{ admits minimizers for any small enough $\tau>0$ and
$\mu\in\PP_2(\rd)$.}
\end{equation}
Indeed, after Remark \ref{lsc} we know that $\mathcal{W}$
satisfies these hypothesis. Hence, we are in the framework of
\cite[¤10.3]{AGS}. We will show how the results therein work for
the case of $\mathcal{W}$.

First of all, we remark that Definition \ref{subdiffdefinition} is
equivalent to the following one.
\begin{pro}\label{otherdefinition}
 $\xxi\in L^2(\rd,\mu;\,\rd)$ belongs to the Wasserstein subdifferential of
 $\mathcal{W}$ at $\mu\in\PP_2(\rd)$ if and only if
\begin{equation}\label{subdifferentialdefinition2}
\w(\nu)-\w(\mu)\geq\int_{\rd}\langle\mathbf{\xxi}(x),
y-x\rangle\,d\gamma(x,y) + o(\mathcal{C}(\mu,\nu;\gamma))
\end{equation}
as $\nu\to\mu$ in $\PP_2(\rd)$, for some optimal transport plan
$\gamma\in\Gamma_o(\mu,\nu)$. Moreover, $\xxi$ is a strong
subdifferential if and only if \eqref{subdifferentialdefinition2}
holds whenever $\nu\to\mu$ in $\PP_2(\rd)$ and
$\Gamma(\mu,\nu)\ni\gamma\to(\id,\id)_\#\mu$ in $\PP_2(\rdrd)$.
\end{pro}
\begin{proof}
 Let $\gamma\in\Gamma(\mu,\nu)$ and define the interpolating curve
$\theta^\gamma(t)=((1-t)\pi^1+t\pi^2)_\#\gamma$ between $\mu$ and
$\nu$, so that $\theta(0)=\mu$ and $\theta(1)=\nu$. We take
advantage of a property of Wasserstein constant speed geodesics,
shown in \cite[Lemma 7.2.1]{AGS}: there exists $\gamma^*$  in
$\Gamma_o(\mu,\nu)$ such that $\Gamma_o(\mu,\theta^{\gamma^*}(t))$
contains a unique element for any $t\in[0,1)$, given by
$\gamma_t:=(\pi^1,(1-t)\pi^{1}+t\pi^{2})_\#\gamma^*$. Then,
\eqref{subdifferentialdefinition2} can be applied in
correspondence of $\gamma_t$ and with $\theta^\gamma(t)$ in place
of $\nu$, and together with \eqref{geoconv}, it gives, for $t\to
0$,
\[
\begin{aligned}
\mathcal{W}(\nu)-\mathcal{W}(\mu)&\geq\frac{\mathcal{W}(\theta^{\gamma^*}(t))-\mathcal{W}(\mu)}{t}+\frac\lambda
2\,(1-t)\mathcal{C}^2(\mu,\nu;\gamma^*)\\
&\geq\int_{\rdrd}\langle\xxi(x),y-x\rangle\,d\gamma^*(x,y)+\frac\lambda
2\,(1-t)\mathcal{C}^2(\mu,\nu;\gamma^*)+\frac1
t\,o(\mathcal{C}(\mu,\theta(t);\gamma_t)).
\end{aligned}
\]
Passing to the limit as $t\to 0$, since
$\mathcal{C}(\mu,\theta(t);\gamma_t)=t^2\mathcal{C}(\mu,\nu;\gamma^*)$,
one gets \eqref{subdifferentialdefinition} for the  plan
$\gamma^*\in\Gamma_o(\mu,\nu)$. One reasons in the same way for the
equivalence in the case of strong subdifferentials: indeed, one
can define $\theta^{\gamma}(t)$ for the generic plan
$\gamma\in\Gamma(\mu,\nu)$ and use the convexity of $\mathcal{W}$
along any interpolating curve $t\mapsto \theta^\gamma(t)$ (see
Proposition \ref{firstproperties}).
\end{proof}

On the other hand, the general definition of subdifferential,
given in \cite[¤10.3]{AGS}, is more technical.
 According to that notion,
the subdifferential is in fact a plan $\beta\in\PP_2(\rdrd)$, as
in the following
\begin{defi}[\textbf{Plan subdifferential}]
Let $\mu\in\PP_2(\rd)$. We say that $\beta\in\PP_2(\rdrd)$ belongs
to the extended subdifferential of $\mathcal{W}$ at $\mu$ if
$\pi^1_\#\beta=\mu$ and there holds
\begin{equation}\label{plan}
\mathcal{W}(\nu)-\mathcal{W}(\mu)\geq\int_{\rdrd\times\rd}\langle
y,z-x\rangle\,d\mbox{\boldmath$\mu$}(x,y,z)+\frac\lambda
2\,\int_{\rdrd\times\rd}|z-x|^2\,d{\mbox{\boldmath$\mu$}}(x,y,z)
\end{equation}
for some
$\mbox{\boldmath$\mu$}\in\mbox{\boldmath$\Gamma$}_o(\beta,\nu)$
(we obtain a strong subdifferential if the inequality holds for
any plan
$\mbox{\boldmath$\mu$}\in\mbox{\boldmath$\Gamma$}(\beta,\nu)$). We
write $\beta\in${\boldmath$\partial$}$\mathcal{W}(\mu)$ (resp.
$\beta\in${\boldmath$\partial$}$_S\mathcal{W}(\mu)$ ). Here the
elements of {\boldmath$\Gamma$}$(\beta,\mu)$ are three-plans, that
is,  measures in $\PP(\rdrd\times\rd)$, such that
$\pi^{1,\,2}_\#${\boldmath$\mu$}$=\beta$ and
$\pi^3_\#${\boldmath$\mu$}$=\nu$. The definition of optimal plan
in this case is
\[
\mbox{\boldmath$\Gamma$}_o(\beta,\nu):=\{\gamma\in\mbox{\boldmath$\Gamma$}(\beta,\nu):\pi^{1,\,3}\gamma\in\Gamma_o(\mu,\nu)\}.
\]
\end{defi}

\begin{remark}\label{rem:subdiff}\rm
Since $\mathcal{W}$ is convex along any linearly
interpolating curve, as noticed in Proposition
\ref{firstproperties}, from \cite[Theorem 10.3.6]{AGS} we learn
that we can equivalently define the extended subdifferential of
$\mathcal{W}$ by asking inequality \eqref{plan} for any
${\mbox{\boldmath$\mu$}}\in{\mbox{\boldmath$\Gamma$}}_o(\beta,\nu)$.
\end{remark}
\begin{remark}\rm
We observe that if $\beta$ is
concentrated on the graph of a vector field $\xxi$, we have
$\beta=(\id,\xxi)_\#\mu$ and in particular $y=\xxi(x)$ for
{\boldmath$\mu$}-a.e. $(x,y,z)\in\rdrd\times\rd$. In this case the
definition reduces to \eqref{subdifferentialdefinition}.
\end{remark}

We recall that, also for extended subdifferentials, there holds
\[
|\partial\mathcal{W}|(\mu)=\mathrm{argmin}
\left\{\left(\int_{\rd}|y|^2\,d\pi^2_\#\beta(y)\right)^{1/2}:\beta\in\mbox{\boldmath$\partial$}\mathcal{W}(\mu)\right\}.
\]
Moreover, the corresponding minimizer is unique. See \cite[Theorem
10.3.11]{AGS}. We denote it by
$\mbox{\boldmath$\partial$}^o\mathcal{W}(\mu)$.

We have the following
\begin{lem}\label{planandbarycenter}
Let $\mu\in\PP_2(\rd)$. The following assertions hold:
\begin{equation}\label{equivalence}
    \beta\in\mbox{\boldmath$\partial$}_S\mathcal{W}(\mu)\Rightarrow\bar{\beta}\in\partial_S\mathcal{W}(\mu)
    \qquad\mbox{and}\qquad
    \beta\in\mbox{\boldmath$\partial$}\mathcal{W}(\mu)\Rightarrow\bar{\beta}\in\partial\mathcal{W}(\mu).
\end{equation}
\end{lem}

\begin{proof}
We begin with the proof for strong subdifferentials.
Let $\beta\in{\boldmath\partial}_S\mathcal{W}(\mu)$ and we write $\beta=\int_{\R^d}\beta_x\,d\mu(x)$.
For any $\gamma\in\Gamma(\mu,\nu)$ there holds
\[\begin{aligned}
\int_{\rdrd}\langle\bar{\beta}(x),z-x\rangle\,d\gamma(x,z)&=
\int_{\rdrd}\left\langle\int_{\R^d}y\,d\beta_x(y),z-x\right\rangle\,d\gamma(x,z)\\
&=\int_{\rdrd\times\rd}\langle y,z-x\rangle\,d\beta_{x}(y)\,d\gamma(x,z).
\end{aligned}
\]
Moreover, taking into account that $\int_{\R^d}d\beta_x(y)=1$ for $\mu$-a.e. $x\in \R^d$, we have that
\[
\int_{\rdrd}|z-x|^2\,d\gamma(x,z)=
\int_{\rdrd\times\rd}|z-x|^2\,d\beta_x(y)\,d\gamma(x,z).
\]
Let us define the three-plan $\mbox{\boldmath$\mu$}$ as
\[
\int_{\rdrd\times\rd}\phi(x,y,z)
d\mbox{\boldmath$\mu$}(x,y,z):=\int_{\rdrd\times\rd}\phi(x,y,z)
d\beta_{x}(y)\,d\gamma(x,z)
\]
for all continuous functions
$\phi:\mathbb{R}^d\times\mathbb{R}^d\times\mathbb{R}^d\to\mathbb{R}$
with at most quadratic growth at infinity. Then,
$\mbox{\boldmath$\mu$}$ belongs to
${\mbox{\boldmath$\Gamma$}}(\beta,\nu)$. Making use of
\eqref{plan} for this particular choice of
{\boldmath$\mu$}$\in\mbox{\boldmath$\Gamma$}(\beta,\nu)$, we see
that $\bar{\beta}$ satisfies \eqref{subdifferentialdefinition}.
Recalling Remark \ref{rem:subdiff}, since $ \gamma\in\Gamma_o(\mu,\nu)\Rightarrow
\mbox{\boldmath$\mu$}\in {\mbox{\boldmath$\Gamma$}}_o(\beta,\nu),
$ reasoning as done for strong subdifferentials the implication
\[
\beta\in{\mbox{\boldmath$\partial$}}\mathcal{W}(\mu)\Rightarrow\bar{\beta}\in\partial\mathcal{W}(\mu)
\]
follows.
\end{proof}

\begin{coro}\label{coroextended}
If $\beta_o$ is the minimal selection in the extended
subdifferential of $\mathcal{W}$ at $\mu$, then it coincides with
 $(\id,\bar{\beta}_o)_\#\mu$, and
$\bar{\beta}_o$ is the minimal selection in $\partial
\mathcal{W}(\mu)$.
\end{coro}
\begin{proof}
   For any $\beta\in$  {\boldmath$\partial$}$\mathcal{W}(\mu)$, using Jensen's inequality, there holds
\[
\int_{\rd}|y|^2\,d\pi^2_\#\beta(y)
=\int_{\rdrd}|y|^2\,d\beta(x,y)=\int_{\rd}\int_{\rd}|y|^2\,d\beta_x(y)\,d\mu(x)\geq\int_{\rd}|\bar{\beta}(x)|^2\,d\mu(x).
\]
This shows that the barycentric projection does not increase the
norm. Therefore, if $\beta_o$ is the minimal selection, since
$\bar{\beta}_o\in\partial\mathcal{W}(\mu)$, there  has to hold
$\beta_o=(\id,\bar{\beta}_o)_\#\mu$. In this case it is also clear
that $\bar{\beta}_o=\partial^o\mathcal{W}(\mu)$.
\end{proof}

\begin{remark}\label{smoothremark}\rm
Because of \eqref{equivalence}, under the same assumptions of
Theorem \ref{smooth} one sees that if
$\beta\in${\boldmath$\partial$}$_S\mathcal{W}(\mu)$, then its
barycenter $\bar{\beta}$ is given by \eqref{smoothgradient}.
\end{remark}

\begin{remark}\label{general}\rm
In the case of strong subdifferentials, the implication of Lemma
\ref{planandbarycenter} is true for any functional $\Phi$
satisfying the assumptions \eqref{hp1} and \eqref{hp2}. It is
shown in Lemma 10.3.4 and Remark 10.3.5 of \cite{AGS} that, given
$\mu\in\PP_2(\rd)$, a minimizer $\mu_\tau$ to
$\Phi(\cdot)+\frac{1}{2\tau}\,d_W^2(\cdot,\mu)$ and a plan
$\hat{\gamma}_\tau\in\Gamma_o(\mu_\tau,\mu)$, there holds
\[
\gamma_\tau\in\mbox{\boldmath$\partial$}_S\Phi(\mu_\tau),
\]
where $\gamma_\tau$ is the rescaled of $\hat{\gamma}_\tau$ (see
Definition \ref{rescaled}). Moreover, among these rescaled plans,
there exists a plan whose barycenter belongs to
$\partial_S\Phi(\mu)$. After Lemma \ref{planandbarycenter}, we may
indeed infer that this holds true for the rescaled of any optimal
plan in $\Gamma_o(\mu_\tau,\mu)$.
\end{remark}

Eventually, we are ready for the proof of Proposition
\ref{convergetominimalselection}. We make use of the general
convergence properties of rescaled plan subdifferentials shown in
\cite[Theorem 10.3.10]{AGS}, passing to barycenters by means of
Lemma \ref{planandbarycenter} and Corollary \ref{coroextended}.

\begin{proof}[\textbf{Proof of {\rm  \textbf{Proposition \ref{convergetominimalselection}}}}]
Let $\tau>0$ be small enough. Once more, let $\mu\in\PP_2(\rd)$,
let $\mu_\tau$ minimize
$\mathcal{W}(\cdot)+\frac{1}{2\tau}\,d_W^2(\cdot,\mu)$ and let
$\hat{\gamma}_\tau\in\Gamma_o(\mu_\tau,\mu)$. Moreover, let
$\gamma_\tau$ be the rescaled of $\hat{\gamma}_\tau$ (as given by
Definition \ref{rescaled}). As a consequence of Theorem
\ref{inequality}, for any $\mu\in\PP_2(\mathbb{R}^d)$ the set
$\mbox{\boldmath$\partial$}\mathcal{W}(\mu)$ is not empty.
Therefore we are in the hypotheses of     \cite[Theorem
10.3.10]{AGS}, which entails, taking into account also
\cite[Remark 10.3.14]{AGS},
\[
\lim_{\tau\to 0}{\gamma}_\tau
=\mbox{\boldmath$\partial$}^o\mathcal{W}(\mu)\quad\mbox{in}\;
\PP_2(\rdrd).
\]
But Corollary \ref{coroextended} implies that
$\mbox{\boldmath$\partial$}^o\mathcal{W}(\mu)=(\id,\partial^o\mathcal{W}(\mu))_\#\mu$.
The convergence above then means that, as $\tau\to0$,
\begin{equation}\label{ultima}
\int_{\mathbb{R}^d\times\mathbb{R}^d}\phi(x,y)\,d(\gamma_\tau)_x(y)\,d\mu_\tau(x)\to
\int_{\mathbb{R}^d\times\mathbb{R}^d}\phi(x,y)\,d(\id,\partial^o\mathcal{W}(\mu))_\#\mu(x,y)
\end{equation}
for any continuous function
$\phi:\mathbb{R}^d\times\mathbb{R}^d\to\mathbb{R}$ with at most
quadratic growth at infinity, where $(\gamma_\tau)_x$ denotes the
family of measures which disintegrates $\gamma_\tau$ with respect
to $\mu_\tau$. Letting $\zeta\in
C^\infty_0(\mathbb{R}^d;\mathbb{R}^d)$, and choosing
$\phi(x,y)=\langle y,\zeta(x)\rangle$ in \eqref{ultima}, we obtain
the convergence in the sense of Definition
\ref{strongconvergencedefinition}  of $\bar{\gamma}_\tau$ to
$\partial^o\mathcal{W}(\mu)$. On the other hand, using Jensen
inequality (in the same way as in the proof of Corollary
\ref{coroextended}) and \eqref{ultima} with $\phi(x,y)=|y|^2$ we
obtain
\[
\limsup_{\tau\to0}\int_{\mathbb{R}^d}|\bar{\gamma}_\tau|^2\,d\mu_\tau\le \lim_{\tau\to0}\int_{\mathbb{R}^d\times\mathbb{R}^d}|y|^2\,d\gamma_\tau=\int_{\mathbb{R}^d\times\mathbb{R}^d}|y|^2\,d(\id,\partial^o\mathcal{W}(\mu))_\#\mu=\int_{\mathbb{R}^d}|\partial^o\mathcal{W}(\mu)|^2\,d\mu,
\]
hence we also have the strong convergence in the sense of Definition \ref{strongconvergencedefinition}.
%
%
%
\end{proof}

\subsection*{Acknowledgements}

The authors would like to thank Giuseppe Savar\'e for several discussions about this work.
JAC acknowledges support from the project MTM2011-27739-C04-02 DGI
(Spain) and 2009-SGR-345 from AGAUR-Generalitat de Catalunya. SL
acknowledges support from Project nr. 25 of the 2007 Azioni
Integrate Italia-Spagna.
SL and EM has been partially supported by the INDAM-GNAMPA project 2011
"Measure solution of differential equations of drift-diffusion,
interactions and of Cahn-Hilliard type".
JAC and SL gratefully acknowledge the
hospitality of the Centro de Ciencias Pedro Pascual de Benasque
where this work was started. EM is supported by a postdoctoral scholarship of the Fondation Math\'ematique Jacques Hadamard, he acknowledges hospitality from Paris-sud University. EM also acknowledges the support from the
project FP7-IDEAS-ERC-StG Grant $\#$200497 (BioSMA).


\end{document}